\documentclass[reqno]{amsart}
\usepackage{a4wide}
\usepackage[english]{babel}
\usepackage{enumerate}
\usepackage{csquotes}
\usepackage[inline,shortlabels]{enumitem}

\usepackage[hyperfootnotes=false]{hyperref} 
\usepackage{color}
\hypersetup{
	colorlinks = true,
	linkcolor = {blue},
	urlcolor = {red},
	citecolor = {blue}
}
\usepackage{cleveref}

\makeatletter
\renewcommand{\tocsection}[3]{
  \indentlabel{\@ifnotempty{#2}{\ignorespaces#1 #2\quad}}\bfseries#3}
\renewcommand{\tocsubsection}[3]{
  \indentlabel{\@ifnotempty{#2}{\ignorespaces#1 #2\quad}}#3}

\newcommand\@dotsep{4.5}
\def\@tocline#1#2#3#4#5#6#7{\relax
  \ifnum #1>\c@tocdepth
  \else
    \par \addpenalty\@secpenalty\addvspace{#2}
    \begingroup \hyphenpenalty\@M
    \@ifempty{#4}{
      \@tempdima\csname r@tocindent\number#1\endcsname\relax
    }{
      \@tempdima#4\relax
    }
    \parindent\z@ \leftskip#3\relax \advance\leftskip\@tempdima\relax
    \rightskip\@pnumwidth plus1em \parfillskip-\@pnumwidth
    #5\leavevmode\hskip-\@tempdima{#6}\nobreak
    \leaders\hbox{$\m@th\mkern \@dotsep mu\hbox{.}\mkern \@dotsep mu$}\hfill
    \nobreak
    \hbox to\@pnumwidth{\@tocpagenum{\ifnum#1=1\bfseries\fi#7}}\par
    \nobreak
    \endgroup
  \fi}
\AtBeginDocument{
\expandafter\renewcommand\csname r@tocindent0\endcsname{0pt}
}
\def\l@subsection{\@tocline{2}{0pt}{2.5pc}{5pc}{}}
\makeatother

\usepackage{amsmath}
\usepackage{amsthm}
\usepackage{amssymb}
\usepackage{mathrsfs} 
\usepackage{mathtools}
\usepackage{nicefrac}

\newcounter{results}[section]

\theoremstyle{plain}
\newtheorem{theorem}[results]{Theorem}
\newtheorem{lemma}[results]{Lemma}
\newtheorem{proposition}[results]{Proposition}
\newtheorem{corollary}[results]{Corollary}

\theoremstyle{remark}
\newtheorem{remark}[results]{Remark}
\newtheorem{example}[results]{Example}

\theoremstyle{plain}
\newtheorem{definition}[results]{Definition}
\newtheorem{assumption}[results]{Assumption}

\numberwithin{equation}{section}

\theoremstyle{remark}
\newtheorem*{continuancex}{Example \continuanceref{} (continued)}
\newenvironment{continuance}[1]
  {\newcommand\continuanceref{\ref{#1}}\continuancex}
  {\endcontinuancex}

\newcommand{\N}{\mathbb{N}}
\newcommand{\R}{\mathbb{R}}
\newcommand{\X}{\mathbb{X}}

\newcommand{\calA}{\mathcal{A}}
\newcommand{\calB}{\mathcal{B}}
\newcommand{\calP}{\mathcal{P}}
\newcommand{\calQ}{\mathcal{Q}}
\newcommand{\calR}{\mathcal{R}}

\newcommand{\calW}{\mathcal{W}}

\newcommand{\rme}{{\mathrm e}}

\newcommand{\x}{\bar x}
\newcommand{\y}{\bar y}

\newcommand{\YY}{X}

\newcommand{\tb}{\textbf}  

\newcommand{\eps}{\varepsilon} 
\newcommand{\de}{{\,\rm d}}
\newcommand{\argmin}{\mathop{\rm argmin}\limits}
\newcommand{\dom}[1]{\mathscr{D}(#1)}

\newcommand{\floor}[1]{\left\lfloor#1\right\rfloor}
\newcommand{\bracket}[1]{\langle #1 \rangle}

\DeclareMathOperator{\TV}{TV}
\DeclareMathOperator{\OT}{OT}
\DeclareMathOperator{\KL}{KL}
\DeclareMathOperator{\Seps}{S_\epsilon}
\DeclareMathOperator{\OTeps}{\OT_\epsilon}
\DeclareMathOperator{\EVI}{{\rm EVI}}

\title[Evolution variational inequalities with general costs]{Evolution variational inequalities with general costs}

\author{Pierre-Cyril Aubin-Frankowski}
\address{Pierre-Cyril Aubin-Frankowski: CERMICS, ENPC, Institut Polytechnique de Paris, France}
\email{pierre-cyril.aubin@enpc.fr}

\author{Giacomo Enrico Sodini}
\address{Giacomo Enrico Sodini: Faculty of Mathematics, University of Vienna, Oskar-Morgenstern-Platz 1, A-1090 Vienna, Austria}
\email{giacomo.sodini@univie.ac.at}

\author{Ulisse Stefanelli}
\address{Ulisse Stefanelli: Faculty of Mathematics, University of Vienna, Oskar-Morgenstern-Platz 1, A-1090 Vienna, Austria, Vienna Research Platform on Accelerating Photoreaction Discovery, University of Vienna, Währingerstrasse 17, 1090 Wien, Austria.}
\email{ulisse.stefanelli@univie.ac.at}
\date{\today}
\keywords{Gradient flows, Evolution variational inequality (EVI), Bregman divergence, Sinkhorn divergence, Kullback–Leibler divergence, Minimizing movements, Variational methods, Cross-convexity}
\subjclass[2020]{49J40, 37L05, 49J27, 49J52}
\begin{document}

\begin{abstract} We extend the theory of gradient flows beyond metric spaces by studying evolution variational inequalities (EVIs)  driven by  general cost functions $c$,  including  Bregman and entropic transport divergences. We establish several properties of the resulting flows for semiconvex functionals, including stability and energy identities. Using novel notions of convexity  related to   costs $c$, we prove that EVI flows are the limit of splitting schemes, providing assumptions for both implicit and explicit iterations.
\end{abstract}
\maketitle
\tableofcontents
\thispagestyle{empty}

\section{Introduction}

\paragraph{\em\bfseries Context} The theory of gradient flows in metric spaces \cite{ags08}, in particular in  spaces of measures, is a  cornerstone  in  the analysis of evolutionary systems.  A reference tool in this context are minimizing movements, relating evolution to incremental minimization. The minimizing-movements approach is however not restricted to metric spaces. Indeed,  it has been  already  applied  out of the metric setting,  including  the case of  Bregman divergences \cite[eq.(1.4)]{Bregman1967} and their mirror flows \cite{NemirovskyYudinBook1983}.  These  appear in entropic regularizations in optimal transport, see \cite{Leonard2014} and \cite{Peyre2019} for reviews.  When leaving the metric setting,  one is confronted with the question of which of the different formulations of gradient-flow evolution, namely, EDI, EDE, or EVI, as discussed in \cite{Ambrosio2012userguide, Santambrogio2017}  and in Section \ref{sec:slope_local},  should be considered. 
\medskip

\paragraph{\em\bfseries Main contributions} In this article, we   investigate the generalization of  Evolution Variational Inequalities (EVIs)  to general costs  on general sets $X$.  This provides a partial extension of the theory of 
\cite{MuratoriSavare}  beyond  the case of the  square distance $d^2$ on a complete metric space $X$. Our  notion of   gradient flow takes the form
\begin{equation}
	\frac{\de^+}{\de t} c(x,x_t) + \lambda c(x,x_t) \le \phi(x)-\phi(x_t) \quad \forall t>0,  x \in \dom{\phi},    \label{eq:evi_diff_intro}
\end{equation}
where  $\lambda\in\R$, $X$ is a set, $c:X\times X\to [0,+\infty)$  is a cost function, $\phi:X\to (-\infty,+\infty]$  with   domain $\dom{\phi}  =\{x\in X : \phi(x)<+\infty \}$. We remark that here $c$ is not necessarily originating from a distance. We call \eqref{eq:evi_diff_intro} EVI and we say that a curve $x:(0,+\infty)\to X$ is an EVI solution and write  
$x(\cdot) \in \EVI_\lambda(X,c, \phi)$ whenever  \eqref{eq:evi_diff_intro} holds for some $\lambda\in\R$. We show in \Cref{thm:equiv} that the differential formulation \eqref{eq:evi_diff_intro} has several equivalent integral expressions.   For symmetric costs $c(x,y)=c(y,x)$, in \Cref{thm:eviprop} we establish that the set $\EVI_\lambda(X,c, \phi)$ of solutions still  enjoys  some of the properties which hold  in metric spaces.  For example, the  EVI  is $\lambda$-contractive and one can prove an energy identity of the form 
    \begin{equation}\label{eq:energy_identity_intro}
	\frac{\de}{\de t}\phi(x_{t+})=-\lim_{h\downarrow 0}\frac{2c(x_t,x_{t+h})}{h^2} =: -  |\dot x_{t+ }|_c^2  ,  
\end{equation}
where $|\dot x_{t+ }|_c$ plays the role of the metric derivative. Symmetry and non-negativity of the cost  are crucially used in order to obtain this. On the other hand, no notion of local slope is  needed.   
 In fact, although  expression \eqref{eq:evi_diff_intro} is global in nature  as it holds for all test points $x$, we discuss in \Cref{sec:slope_local} a local formulation of the form  
$\nabla_{2,1} c(x_{t}, x_{t})\dot x_{t} \in \partial \phi(x_{t})$, where $\nabla_{2,1}$ is the mixed-Hessian and $\partial \phi$ is the Fr\'echet subdifferential. 

For geodesic metric spaces $(X,d)$ and $c={d^2}/{2}$, the EVI \eqref{eq:evi_diff_intro}  has been related to some notion of  
convexity of $\phi$ w.r.t.\ $d^2$,  see  \cite[Section 3.3]{MuratoriSavare}.  Moreover, in metric spaces one may prove existence of solutions to \eqref{eq:evi_diff_intro} via  minimizing movements and  the  implicit Euler  scheme.  

For general sets $X$ and  general  costs $c$,  a corresponding notion of convexity has been recently brought to evidence in \cite{leger2023gradient} by the first-named author. This notion takes the name of {\it cross-convexity} and delivers an extension of  usual convexity. For $c={d^2}/{2}$,  cross-convexity  takes the form of a discrete EVI, as in  \cite[Corollary 4.1.3]{ags08}. For $c$  being  a Bregman divergence, cross-convexity corresponds to the so-called   {\it three-point inequality} in mirror descent, see, e.g., \cite[Lemma 3.2]{chen1993convergence}.  In essence, cross-convexity  can be interpreted as a compatibility property between energy and cost,  see \Cref{sec:compatibility}.  In the cross-convex setting,  existence of a solution $x(\cdot)$ to \eqref{eq:evi_diff_intro} for a given initial point $x_0$  can be  obtained  by  taking the limit $\tau\to 0$ in the following alternating minimization (see \eqref{eq:iteration_y}--\eqref{eq:iteration_x}) based on the $c/\tau$-transform $f^{c/\tau}(y):=\sup_{x'\in X}[f(x')-\frac{c(x',y)}{\tau}]$, 
\begin{align}
	y_{i+1}^{\tau}&\in \argmin_{y \in X} \left\{ \frac{c(x_{i}^{\tau},y)}{\tau}+g(x_{i}^{\tau}) + f^{c/\tau}(y)\right\}, \label{eq:iter_y_intro} \\
	x_{i+1}^{\tau}&\in\argmin_{x \in X} \left\{ \frac{c(x,y_{i+1}^{\tau})}{\tau} +g(x)+ f^{c/\tau}(y_{i+1}^{\tau})\right\}.\label{eq:iter_x_intro}
\end{align}
 These iterations correspond to  a splitting scheme over $\phi=f+g$, with one explicit step on $f$ and one implicit on $g$. Schemes of this form were extensively studied in \cite{leger2023gradient}. Under some cross-convexity requirements  (see  \Cref{ass:ass2}), for $c$ symmetric,  a  lower-bounded $\phi$, and $\tau$ small enough, we show in \Cref{thm:existence_limit_scheme_fg} the following error estimate between the discrete-time scheme and its continuous-time limit:
\begin{equation}\label{eq:error_estimate_intro}
	c \left (\bar{x}^{\tau}_t, x_t \right ) \le 2\tau (\phi(x_0)-\inf \phi) \quad \forall t \ge 0,
\end{equation}
 where $\bar{x}^{\tau}(\cdot)$ is the piece-wise constant in time interpolant of the discrete values $\{x^\tau_i\}$ on the uniform partition $\{i\tau\}$. 

We now list three examples of costs $c$ of interest.

\begin{example}(distances, $d^p$)
	\label{ex:distance}  Let $(X,d)$ be a complete metric space and $c=d^p$ for $p\ge 1$. Then, for $f=0$, \eqref{eq:iter_x_intro} is the usual implicit Euler discretization.  Flows with asymmetric distances were also considered in \cite{Rossi2009,Chenchiah2009,Ohta2023}.
\end{example}

 Similarly to \cite{ags08},  a key example in the following is  
 $X=\calP(\X)$,  that is  the  set  of  non-negative Borel measures over the measurable space  $\X$ with total mass 1.  Our analysis  covers  in particular that of \cite{ags08} if the cost $c$ is the Wasserstein-2 distance. Two other prominent costs, on which we will showcase our assumptions and theorems,  are  the Kullback--Leibler and  the  Sinkhorn divergences.
\begin{example}(Kullback--Leibler divergence, $\KL$)
	\label{ex:KL} The {\it Kullback--Leibler divergence} or {\it relative entropy} over probability measures is
	\begin{equation*}
		\KL(\mu|\bar \mu) = \left\{
		\begin{array}{ll}
			\int_{\X}\log \left(\frac{\de \mu}{\de \bar \mu}(x)\right) {\rm d}  \mu(x)& \mbox{if } \mu \ll \bar \mu ,  \\
			+\infty & \mbox{otherwise ,  }
		\end{array}\right.
	\end{equation*}	
	where, for any $\mu,\nu\in \calP(\X)$, we write $\mu \ll \nu$ when $\mu$ is absolutely continuous w.r.t\ $\nu$, i.e., when it  admits  a Radon--Nikodym derivative $ {\rm d}  \mu/ {\rm d}  \nu$. The $\KL$ is a Bregman divergence of the entropy $\KL(\cdot| \rho)$ where $ \rho $ is  interpreted as  a positive reference  measure (the Lebesgue measure  on  $\X\subset \R^d$,  for instance). Note that  $\KL$ is not symmetric and does not satisfy the  triangle  inequality. 
\end{example}

\begin{example}(Sinkhorn divergence, $\Seps$)
	\label{ex:Sinkhorn} Assume $\X\subset \R^n$ to be compact. Fix $\epsilon>0$ and take a ground cost $c_\X\in C^1(\X\times\X;\R)$ such that $\exp(-c_\X/\epsilon)$ is a positive definite and universal reproducing kernel, see \cite{feydy2019interpolating}. Then, the {\it entropic optimal transport (EOT) dissimilarity} is defined as
	\begin{equation}\label{eq:oteps}
		\OTeps(\mu,\nu)=\min_{\pi \in \Pi(\mu,\nu)} 
		\int_{\X\times\X} c_\X(x,y) d\pi(x,y)+\epsilon\KL(\pi|\mu\otimes \nu)=\eps\KL(\pi|e^{-c_\X/\eps}\mu\otimes\nu),
	\end{equation}
	where we denote by $\Pi(\mu,\nu)$ the set of couplings having first marginal $\mu$ and second marginal $\nu$.  As $\OTeps(\mu,\mu)\not =0$ in general \cite[Section 1.2]{feydy2019interpolating}, the {\it Sinkhorn divergence} was defined in \cite{genevay2018learning} as  
	\begin{equation}\label{eq:s_eps}
		\Seps(\mu,\nu)=\OTeps(\mu,\nu)-\frac12\OTeps(\mu,\mu)-\frac12\OTeps(\nu,\nu)
	\end{equation}
     which indeed fulfills $\Seps(\mu,\mu)=0$. 
	We refer to the introduction of \cite{Carlier2017}  for a thorough presentation  of $\OTeps$ and to that of \cite[p.~5]{lavenant2024riemannian} for $\Seps$. A self-contained discussion on the limiting behaviours for $\epsilon\to 0$ or $\epsilon\to \infty$ can be found in \cite{Neumayer2021}. As evidenced in \cite[Section 7.1]{lavenant2024riemannian},  neither  $\Seps$ nor $\sqrt{\Seps}$  satisfy the triangle inequality, even though $\Seps$ is symmetric and metrizes the convergence in law  \cite[Theorem 1]{feydy2019interpolating}.
\end{example}
Other natural examples of costs include doubly nonlinear evolution equations  in a vector space  $X$, $\Psi:X\to \R$, and $c_\tau(x,y)=\tau\Psi((x-y)/{\tau})$,  see   \cite{Rossi2009}.  In our framework, we can only deal with  
 homogeneous   potentials  $\Psi$.  In fact, existence of formulations of the form \eqref{eq:evi_diff_intro} for $\Psi$ not homogeneous do not presently seem available.    

Due to the variety of the applications in machine learning, e.g.\ learning on graphs, several directions have been pursued  departing  from the Euclidean-distance  setting.  For instance, these new non-metric costs can be inspired by optimal transport and based on Lagrangians \cite{pooladian2024neural}, or on elastic nets \cite{klein2024learning}. Sometimes  as in \cite{auffenberg2025unsupervised}, the new cost is learnt, e.g.\ via neural or representation-based ground costs. With the emergence of many tasks of learning where $X$ is a space of probabilities, or of probabilities on probabilities, we  foresee  that  non-metric costs, e.g.\ $f$-divergences such as the $\chi^2$, Bregman divergences, or hybrid transport–reaction functionals,  will   become more common, although metrics as in \cite{cuturi2014ground} will remain prominent.

\medskip

\paragraph{\em\bfseries Related work}  Reference works on EVIs in metric spaces are the  three great expositions  in   \cite{ags08,Ambrosio2012userguide,MuratoriSavare}.  Notably, our theory is designed to recover some of these metric results in the case of $(X,d)$ being  a complete metric space and $c={d^2}/{2}$.  The  continuous setting of \Cref{sec:EVI}  is somewhat   closer to \cite[Section 3]{MuratoriSavare},  whereas the  discrete setting of \Cref{sec:f+g}  takes  inspiration from \cite[Section 3.2.4]{Ambrosio2012userguide},  especially in its use of dyadic partitions.  

Several extensions of the canonical gradient flows in metric settings have  already   been considered.  In particular, \cite{Ohta2023}  deals with  asymmetric distances and EDEs on Finsler manifolds,  and  \cite{Craig2017}  studies  EVIs when $\lambda d^2$ in \eqref{eq:evi_diff_intro} is  replaced by $\lambda \omega(d^2)$  for  $\omega:[0,+\infty)\to [0,+\infty)$,  as  motivated by  the celebrated  Osgood's criterion. Beyond  metric spaces,  \cite{Rossi2024} developed a theory for minimizing movements  under so-called  {\it action} costs $a(\tau,x,y)$.  This theory actually covers our choice ${c(x,y)}/{\tau}$, stopping short of discussing the EVI formulation, nonetheless.  

Compactness and completeness are  the two alternative tools which are used to ascertain the existence of minimizing-movement limiting curves. Preferring one over the other naturally brings to two alternative arguments, see e.g.~\cite{rankin2024jkoschemesgeneraltransport}.  Specifically, \cite{MuratoriSavare}  is exclusively using  completeness  whereas \cite{Ohta2023} is based  on compactness.  In our analysis, we elaborate on both approaches, by  providing corresponding sets  
of assumptions.   

In  the metric setting, a  discrete implicit-implicit splitting scheme based on the discrete EVI  has been advanced in \cite[Eq.~(1.1)]{Clement2011}.   An explicit Euler scheme to deal with evolutions in the space of probability measures has been studied in \cite{CavagnariSavareSodini22}, where also the notion of EVI solution is used. 
To our best knowledge,  we provide here the first theory for both implicit and explicit schemes in this EVI context  and in such generality. 

\medskip

\paragraph{\em\bfseries Plan of the paper}   \textbf{Section \ref{sec:EVI}}    is devoted to  formulate the EVI  
with general costs: in Section \ref{sec:comp} we discuss the topological framework we adopt, and in particular the interplay between the cost function and the underlying topology.  In  Section \ref{sec:efevi}, we present the equivalent formulations of the EVI and  discuss the properties of EVI solutions.    Section \ref{sec:slope_local}  discusses  possible local formulations of the EVI,  as well as how the EVI in the metric case leads to the metric slope, and hence to the EDE and EDI. In \textbf{Section \ref{sec:f+g}}, we discuss the existence of EVI solutions for $\phi=f+g$ via an alternating minimization scheme.  This is done first  in  Section \ref{sec:f=0}  in the specific case of the implicit scheme when $f=0$, and then in full generality in Section \ref{sec:f+gproof}.  Prior to this, sufficient compatibility conditions between the energy and cost, extending those from the metric setting,  are discussed in Section \ref{sec:compatibility}.

\section{Evolution Variational Inequalities with general cost \texorpdfstring{$c$}{c}}\label{sec:EVI}

\paragraph{\em\bfseries Basic notation}
In the following, we set $\R_{>0}=(0,+\infty)$.  We use the shorthand l.s.c.\ for lower semicontinuous. Given the real function $u:\R_{>0} \to \mathbb{R}$, we indicate its {\it right derivative} at $t>0$ (whenever existing) as $u'(t+)$. The {\it right Dini derivatives} are denoted by 
$$\frac{\rm d^+}{\rm d t}u(t)=\limsup_{s\downarrow t}\frac{u(s)-u(t)}{s-t},\quad \frac{\rm d}{\rm d t^+}u(t)=\liminf_{s\downarrow t}\frac{u(s)-u(t)}{s-t}.$$

Letting $X$ be a nonempty set, we indicate a trajectory $x:\R_{>0}\to X$ with the symbol $x(\cdot)$ or $(x_t)_{t>0}$. The point at time $t>0$ on the trajectory is denoted by $x_t$. The image of the map $x$ over $I \subset \R_{>0}$ is indicated by $\{x_t\}_{t\in I}$. Given a function $\phi:X \to (-\infty,+\infty]$, we set $\dom{\phi}=\{x\in X\::\: \phi(x)<+\infty\}$ to be its {\it domain} and we say that $\phi$ is {\it proper} if $\dom{\phi}\not = \emptyset$.

Given a {\it directed set} $\mathcal{A}$ (i.e., a nonempty set with a preorder $\leq$  such that every pair of elements has an upper bound), we recall that a {\it net} in $X$ is a mapping $x:\mathcal{A} \to X$, which we also indicate as $(x_a)_{a\in \mathcal A}$. We say that $(y_b)_{b\in \mathcal B}$ is a  subnet   of $(x_a)_{a\in \mathcal A}$  if  there exists a non-decreasing {\it final} function $J:\mathcal B \to \mathcal A$ such that $y_b=x_{J(b)}$ for all $b\in \mathcal B$. We use the symbol $\sigma{-}\lim_a x_a$ to indicate the limit of $(x_a)_{a\in \mathcal A}$ with respect to the topology $\sigma$.  We also write $x_\alpha \xrightarrow[]{\sigma} x$. 
\medskip

\paragraph{\em\bfseries Costs}
We assume to be given a {\it cost} $c:X\times \YY \to \mathbb{R}$. We say that the cost is {\it dissipative} if  
    \begin{equation}\label{eq:addass}\tag{Diss}
     c \ge 0 \quad \text{ and } \quad c(x,x')=0 \text{ if and only if } x=x'.
    \end{equation}
    This in particular entails the dissipativity of both discrete-time and continuous-time evolutions driven by 
    $\phi=f+g$.  
    More precisely,  in  \Cref{sec:f+g}, for splitting schemes on $\phi=f+g$, non-negativity is used to weave the two iterations together. For  the  explicit scheme, i.e.\ $g=0$, non-negativity is not necessary since $c$-concavity  ensures that $\phi=f$ decreases. Costs fulfilling \eqref{eq:addass} that are symmetric have some surprising connections with notions of tropical monotonicity \cite[Proposition 4.1]{aubin2022tropical}. As also made clear in \Cref{sec:f+g}, we often  handle specific expressions, in which differences of cost appear,  so  that replacing the lower bound $0$ in \eqref{eq:addass} by any other constant is admissible.  We restrict ourselves to finite-valued costs in order to avoid subtleties when treating differences of costs. Extended-valued costs may also be considered at the expense of a more intricate analysis.

\subsection{Conditions for compatibility of the topology \texorpdfstring{$\sigma$}{s} with \texorpdfstring{$(c, \phi)$}{c,p}}\label{sec:comp} To discuss convergence and continuity  we need to introduce a topology $\sigma$ on the space $X$. This topology $\sigma$ will be required to be compatible with the cost $c$ and the energy $\phi$. We start by specifying such compatibility requirement in Definition \ref{def:comp}, providing later some sufficient conditions  for such compatibility to hold.  

\begin{definition}[Compatible topology]\label{def:comp} Given $c$ satisfying \eqref{eq:addass} and  $\phi:X \to (-\infty,+\infty]$,   we say that a Hausdorff topology $\sigma$ on $X$ is \emph{compatible with the pair $(c,\phi)$}  if
\begin{enumerate}
    \item[\rm (1)]\label{it:sigma_c_conv} $c$-convergent nets in sublevels of $\phi$ are $\sigma$-convergent: if $(x_\alpha)_{\alpha\in \calA} \subset \{ \phi \le r  \}$ for some $r  \in \R$, $x \in X$, and $c(x_\alpha, x) \to 0$ or $c(x,x_\alpha) \to 0$, then $x_\alpha \xrightarrow[]{\sigma} x$.
    \item[\rm (2)]\label{it:sigma_c_g_lsc} $\phi$ is $\sigma$-lower semicontinuous and $c$  is jointly $\sigma$-lower semicontinuous.
\end{enumerate}
We say that  $\sigma$ is \emph{forward-Cauchy-compatible with the pair $(c,\phi)$} if furthermore 
\begin{enumerate}
    \item[\rm (3)] $c$-forward-Cauchy sequences in sublevels of $\phi$ are $\sigma$-convergent: if $(x_n)_n \subset \{ \phi \le r  \} \subset X$ for some $r  \in \R$ satisfies
    \[ \forall \eps >0  \,\exists N_{\eps} 
    \in \N \text{ s.t. } c(x_n, x_m) < \eps \quad \forall   n \ge m > N_\eps ,\]
    then there exists $x\in X$ such that $x_n \xrightarrow[]{\sigma} x$.
    \end{enumerate}
\end{definition}

Note that we do not require $c$ to be symmetric. In particular,  condition (1) above  allows for situations where $c(x_\alpha, x) \to 0$ but $c(x, x_\alpha) \not\to 0$. When $c$ is symmetric,  we simply refer to  $c$-Cauchy sequences, omitting the term {\it forward}.  
As  the topology $\sigma$ is not assumed to be metrizable,  compactness
does not imply sequential compactness,  which forces us  in the following  to   deal with nets instead of sequences.

\begin{continuance}{ex:distance} Let $\omega: [0,+\infty) \to [0,+\infty)$ be a non-negative, strictly increasing function such that $\omega(0)=0$. Let $(X, d)$ be a complete metric space with induced topology $\sigma_d$, and let $c:= \omega(d)$. If $\phi$ is $\sigma_d$-lower semicontinuous, then $\sigma_d$ is compatible with $(c,\phi)$ and $c$ is jointly $\sigma_d$-continuous. In \cite{Chenchiah2009,Rossi2009,Ohta2023} asymmetric distances are also considered: assume $(X, d)$ is an asymmetric metric space (that is, $d$ satisfies all the axioms of a distance but symmetry) and let $\sigma$ be the topology induced by the forward balls (cf.~\cite[Definition 2.1]{Ohta2023}, \cite[Definition 2.1]{Chenchiah2009}). Then, the assumptions 4.2, 4.3, 4.7 of \cite{Chenchiah2009} on $\sigma$ and $\phi$  are fulfilled in  our setting, that is, $\sigma$ is forward-Cauchy compatible with $(c,\phi)$.
\end{continuance}

\begin{continuance}{ex:KL}
    For $X\subset\calP(\X)$ with $\X$  being  a Polish space  and  $c=\KL$, $\sigma$ can be chosen to be the weak topology on measures which is dominated by the strong topology induced by the $\TV$ distance. By Pinsker's inequality, 
    valid on measurable spaces \cite[Theorem 31]{van2014renyi}, $\KL$ dominates the $\TV$ distance,  whence {\rm (1)} is satisfied. 
    Moreover, it  is known that $\KL$ is jointly $\sigma$-lower semicontinuous \cite[Theorem 19]{van2014renyi}, so that  {\rm (2)} holds for all $\phi$ which are $\sigma$-l.s.c. To ensure finite values for $c$, we can take $X$ to be a $\sigma$-closed subset of $\calP(\X)$ such that $\KL(x|y)<+\infty$ for all $x,y\in X$, e.g., fixing a measure $\rho$ and setting $\mu\in X$ if and only if $\mu =f\rho$ with $f\in [a,b]$ $\rho$-a.s.\ with $0<a<b<+\infty$. In the latter case, $X$ is naturally identified with a subset of $L^1(\rho)$ and one could consider on $X$ the $L^1(\rho)$-topology $\sigma_\rho$ which is stronger than $\sigma$. Consequently, $\KL$ is jointly $\sigma_\rho$-continuous on $X$, so that $\sigma_\rho$ is compatible with $(c, \phi)$ for any $\sigma_\rho$-lower semicontinuous functional $\phi:X \to (-\infty, + \infty]$.
\end{continuance}

\begin{continuance}{ex:Sinkhorn} 
    For $X=\calP(\X)$ with compact $\X$, $c=\Seps$ and $\sigma$ the weak topology on measures, by \cite[Theorem 1]{feydy2019interpolating}, $\Seps$ is  jointly  $\sigma$-continuous, so ({\rm 2)} holds  whenever  $\phi$ is   $ \sigma$-l.s.c.    Moreover, $\Seps$ metrizes  the  weak convergence, hence {\rm (1)} is satisfied.    Since $X=\calP(\X)$ is $\sigma$-compact,  owing to \Cref{le:topology}  below one has that forward-Cauchy-compatibility also holds.  
\end{continuance}
We  now present some sufficient  conditions   for the compatibility of $\sigma$ with the  
 pair  $(c,\phi)$. The first condition  applies to the case of a function $\phi$ with  $\sigma$-compact  sublevels.  

\begin{lemma}[$\phi$ with $\sigma$-compact sublevels]\label{le:topology}  Let $c$ be a cost on $X$ satisfying \eqref{eq:addass}, let  $\sigma$ be a Hausdorff topology on $X$ such that  $c$ is jointly $\sigma$-lower semicontinuous, and   such that  the  sublevels of $\phi:X\to(-\infty,+\infty]  $ are $\sigma$-compact.  Then, $\sigma$ is forward-Cauchy-compatible with $(c,\phi)$.
\end{lemma}
\begin{proof} Since $\sigma$ is Hausdorff,  the  sublevels of $\phi$ are closed, so that $\phi$ is $\sigma$-lower semicontinuous  and  property (2) of Definition \ref{def:comp} is satisfied. 

Let us show that property (1) holds: let $(x_\alpha)_{\alpha\in \calA} \subset X$ be a net in a sublevel of $\phi$ and $x \in X$ be such that $c(x,x_\alpha)\to 0$. Let $(x_\beta)_{\beta \in \calB}$ be any subnet of $(x_\alpha)_{\alpha\in \calA}$  which is  $\sigma$-converging to a point $y\in X$. By  joint   $\sigma$-lower semicontinuity of $c$, we have
\[ 0 = \liminf_{\beta \in \calB} c(x, x_\beta) \ge c(x,y)\]
which implies that $x=y$ by   \eqref{eq:addass}. This, together with the $\sigma$-compactness of sublevels of $\phi$, shows that the net $(x_\alpha)_{\alpha\in \calA}$ $\sigma$-converges to $x$.

 In order to check the  $c$-forward-Cauchy  property (3) of Definition \ref{def:comp},  let $(x_n)_n \subset \{ \phi \le r \} \subset X$ for some $r  \in \R$ and suppose that $(x_n)_n$ is $c$-forward-Cauchy.  As the sublevel $\{ \phi \le r \} $ is $\sigma$-compact, we can find a $\sigma$-convergent subnet $(x_{J(\lambda)})_{\lambda \in \Lambda}$, where  $\Lambda$ is a directed set  and  $J:\Lambda \to \N$ is a final  monotone   function,  namely,  
$\sigma{-}\lim_\lambda x_{J(\lambda)} =x$ for some $x \in X$. 
Fix $\eps>0$ and let $N_\eps \in \N$ be such that $c(x_n, x_m) < \eps$ for every  $n \ge m \ge N_\eps$. Since $J$ is a final function, we can find some $\lambda_{\eps,m} \in \Lambda$ such that $\lambda \ge m\ge N_{\eps}$ for every $\lambda \ge \lambda_{\eps,m}$. Thus we have
\[ c(x_{J(\lambda)}, x_{m}) < \eps \quad \text{ for every } m \ge N_\eps, \, \lambda \ge \lambda_{\eps,m}.\]
Passing first to the $\liminf_\lambda$ and using the  joint  $\sigma$-lower semicontinuity of $c$, we get that 
\[ c(x, x_m) \le \eps \quad \text{ for every } m \ge N_\eps,\]
which means that $x_n$ $c$-converges to $x$.  By the first part of this proof, we deduce that $x_n$ $\sigma$-converges to $x$,  which proves the assertion. 
\end{proof}

Alternatively to checking that a given topology is compatible, one can use the  
cost $c$ 
to  construct  one. Given a  symmetric cost $c$ satisfying \eqref{eq:addass},  referred to as  a {\it semimetric} in part of the literature, see \cite[Section 2]{semim2}, one can define  a topology $\sigma_c$ on  $X$ by prescribing that a set $A\subset X$ is {\it open} (hence, $A\in \sigma_c$)  
if for every $x \in A$ there exists $r>0$ such that
\[ B_c(x,r):=\{y \in X : c(x,y) < r \} \subset A. \]
Notice that in general $\sigma_c$ is not Hausdorff, $B_c(x,r) \notin \sigma_c$, $c$-convergent sequences do not have the $c$-Cauchy property, and $c$ is not  jointly  $\sigma_c$-continuous.   However, if $c$ is a \emph{regular semimetric}  in the sense of Lemma \ref{lem:semi} below,  these pathologies do not appear.   

\begin{lemma}[$c$ regular semimetric]\label{lem:semi} Suppose that  $c$ is a symmetric cost satisfying \eqref{eq:addass} and that it is additionally  \emph{regular}, i.e.
\begin{equation}\label{eq:semi} \lim_{r \to 0 } \sup_{y \in X} \sup_{z,w \in B_c(y,r)} c(z,w)=0,
\end{equation}
or, equivalently, that $\Phi_c:\R_{>0}\times \R_{>0} \to \R_{>0}$ defined as
\begin{equation}\label{eq:semi2}
    \Phi_c(r_1,r_2):=\sup \{ c(x,y) \, | \, \exists \, p \in X: c(p,x) \le r_1, \, c(p,y)\le r_2\}
\end{equation}
is continuous at $(0,0)$.  Moreover, assume that   $(X,c)$ is $c$-complete (i.e.,~$c$-Cauchy sequences are $c$-convergent) and $\phi$ is $\sigma_c$-lower semicontinuous.  Then,   $\sigma_c$ is compatible with $(c,\phi)$ and $c$ is jointly $\sigma_c$-continuous.
\end{lemma}
\begin{proof} The equivalence  between \eqref{eq:semi} and \eqref{eq:semi2}   can be found in \cite[Lemma 1]{bessenyei2014contractionprinciplesemimetricspaces}. By \cite[Theorem 3.2]{semim}, the regularity and completeness of $c$ imply the existence of a complete metric $\varrho$ on $X$ which is uniformly equivalent to $c$, i.e.,
\begin{enumerate}
    \item for every $\eps>0$ there exists $\delta>0$ such that $\varrho(x,y)< \eps$ whenever $c(x,y)< \delta$;
    \item for every $\eps>0$ there exists $\delta>0$ such that $c(x,y)< \eps$ whenever $\varrho(x,y)< \delta$.
\end{enumerate}
This gives in particular that $\sigma_c$ is induced by $\varrho$ so that $\sigma_c$ is Hausdorff, $\varrho$ is $\sigma_c$-continuous (hence also $c$ is), and that $c$-convergent sequences are $\varrho$-convergent, hence also $\sigma_c$-convergent.
\end{proof}

\subsection{Equivalent formulations and main properties of the EVI}\label{sec:efevi}

 In this section, we introduce the notion of EVI solution and discuss some of its properties. To start with, we present a result on equivalent characterizations of trajectories.

\begin{theorem}[Equivalent  EVI definitions]\label{thm:equiv}  Let $\lambda \in \R$, $\phi:X\to (-\infty, + \infty]$,  $c$  satisfies  \eqref{eq:addass},  and the topology  $\sigma$ be compatible with the pair $(c,\phi)$.  Given the trajectory $x(\cdot):\R_{>0}  \to X$,  for either $\lambda \ge 0$ or $\lambda < 0$ and $c$  additionally  separately $\sigma$-continuous, the following are equivalent:
\begin{enumerate}[\rm i)]
    \item\label{it:evi_diff} $x(\cdot)$ is $\sigma$-continuous,  $x_t \in \dom{\phi}$ for all $t>0$,  and $x(\cdot)$ satisfies the \emph{EVI in differential form}:
    \begin{equation}
    \frac{\de^+}{\de t} c(x,x_t) + \lambda c(x,x_t) \le \phi(x)-\phi(x_t) \quad \forall t>0,  x \in \dom{\phi}.    \label{eq:evi_diff}
    \end{equation}
    \item\label{it:evi_int} For all $x\in \dom{\phi}$, the maps $t\mapsto \phi(x_t)$ and $t\mapsto c(x,x_t)$ belong to $L^1_{ \rm loc}(\R_{>0})$, $(s,t)\mapsto c(x_s,x_t)$ is Lebesgue measurable on  $\R_{>0}\times \R_{>0}$,   and $x(\cdot)$ satisfies the \emph{EVI in integral form}
    \begin{equation}
    c(x,x_t)-c(x,x_s) + \lambda\int_{s}^t c(x,x_r) \de r \\
    \le (t-s)\phi(x)-\int_s^t \phi(x_r) \de r \quad \forall \, 0 < s \le t.    \label{eq:evi_int}
    \end{equation}
    \item\label{it:evi_exp_int} $x(\cdot)$ satisfies the \emph{EVI in exponential-integral form}
    \begin{equation}
    e^{\lambda(t-s)}c(x,x_t)-c(x,x_s)
    \le E_\lambda(t-s)\left(\phi(x)-\phi(x_t)\right) \quad \forall \, 0 < s \le t,   x \in \dom{\phi},    \label{eq:evi_exp_int}
    \end{equation}
    where $E_\lambda(t)\coloneqq \int_0^t e^{\lambda r}  \de  r = \left\{ \begin{array}{cl}
          \frac{e^{\lambda t}-1}{\lambda}  &\text{ for } \lambda\neq 0 \\
          t &\text{ for } \lambda = 0.
    \end{array} \right.$
\end{enumerate}
    Furthermore, if $x(\cdot)$ satisfies any of the above, then $t\mapsto \phi(x_t)$ is lower semicontinuous and non-increasing (in particular, right-continuous), and we have the oriented  local Lipschitzianity relation
    \begin{equation}\label{eq:Lip_c}
        0 \le c(x_s,x_t)\le E_{-\lambda}(t-s)(\phi(x_{s_0})-\phi(x_{t_0})) \quad \forall \, 0<s_0 \le s \le t \le t_0. 
    \end{equation}
     \end{theorem}

\begin{remark}[On the EVI formulations and the estimates] Arguably, the differential \eqref{eq:evi_diff} and  the  integral \eqref{eq:evi_int} formulations of the EVI are the natural generalization of the classical formulations of the metric case  in  \cite[Theorem 3.3]{MuratoriSavare}. The one with the exponential \eqref{eq:evi_exp_int} is nevertheless very  useful  to  obtain   estimates. Its drawback is that it misses the energy identity by a factor $2$, as easily seen when taking $\lambda=0$ and $x=x_s$ in \eqref{eq:evi_exp_int}, or noticing that for $x=x_t$ the  left-hand side in non-positive.   As a corrective, for symmetric costs, the  missing  term is expressed in \eqref{eq:evi_exp_int_ref} below. Notice also that the formulation in \eqref{eq:evi_exp_int} does not depend on the topology $\sigma$. Finally, we highlight that we used the term ``local Lipschitzianity" to describe \eqref{eq:Lip_c} even if the cost $c$ may not be a distance; in the metric case, this would indeed give that the curve $x(\cdot)$ is locally Lipschitz and, in particular, locally absolutely continuous.
\end{remark}

\begin{remark}[Separate and joint lower semicontinuity] It  would   also  be  possible to formulate and prove Theorem \ref{thm:equiv}  by  assuming the cost $c$ to be merely separately $\sigma$-lower semicontinuous in the case $\lambda \ge 0$, discarding the joint measurability of $(s,t)\mapsto c(x_s,x_t)$ in \ref{it:evi_int}. Indeed, the only point where this is used is in \eqref{eq:evi_int_applied3} when writing a double integral, and the latter can be dropped using that $\lambda c\ge 0$ when $\lambda \ge 0$. In the case $\lambda<0$, one should nevertheless assume $c$ to be additionally jointly $\sigma$-continuous, in order for the full proof to work.
\end{remark}

\begin{proof} The proof  follows the blueprint of  
\cite[Theorem 3.3]{MuratoriSavare}, the main difference being that here we do not assume that $c$ is  jointly $\sigma$-continuous, but merely   jointly  
$\sigma$-l.s.c.\ and we do not use the triangle inequality. This calls for some adaptation of the argument. Along this proof, we use the short-hand notation $\tilde \phi = \phi \circ x$. Before entering  into   the details of the proof, we outline here the arguments used the achieve the different implications:
\begin{enumerate}
    \item \ref{it:evi_diff} $\Rightarrow$ \ref{it:evi_int}: this is based on a simple application of \cite[Lemma A.1]{MuratoriSavare} to transform the distributional/differential  formulation  \ref{it:evi_diff} into the integral   formulation  \ref{it:evi_int}.
    \item \ref{it:evi_exp_int} $\Rightarrow$ \ref{it:evi_diff}: we first use \eqref{eq:evi_exp_int} together with \eqref{eq:addass} and the compatibility of $c$ with $\sigma$ in \Cref{def:comp} to derive the monotonicity of $\tilde \phi = \phi \circ x  $ and the right $\sigma$-continuity of $x(\cdot)$; then the validity of \eqref{eq:evi_diff} follows directly from differentiating \eqref{eq:evi_exp_int}. Finally the left $\sigma$-continuity of $x(\cdot)$ follows again by a manipulation of \eqref{eq:evi_exp_int}, the compatibility of $c$ with $\sigma$, and the monotonicity of $\tilde \phi$.
    \item \ref{it:evi_int} $\Rightarrow$ \ref{it:evi_diff}: we repeat the first part of the proof of the implication \ref{it:evi_exp_int} $\Rightarrow$ \ref{it:evi_diff} using \eqref{eq:evi_int} instead of \eqref{eq:evi_exp_int} to show the right $\sigma$-continuity of $x(\cdot)$. Integrating \eqref{eq:evi_int}, we prove that $\tilde \phi$ is non-increasing when restricted to its Lebesgue points; the extension to all points is then obtained using the $\sigma$-lower semicontinuity of $\phi$, again \eqref{eq:evi_int}, and the joint $\sigma$-lower semicontinuity of $c$ in \Cref{def:comp}. Finally, we use the monotonicity of $\tilde \phi$ to show the left $\sigma$-continuity of $x(\cdot)$ and that \eqref{eq:evi_int} gives \eqref{eq:evi_diff}.
    \item \ref{it:evi_diff} $\Rightarrow$ \ref{it:evi_exp_int}: after multiplying \eqref{eq:evi_diff} by $e^{\lambda t}$, the resulting expression is integrated again through \cite[Lemma A.1]{MuratoriSavare} to obtain \eqref{eq:evi_exp_int}, also using that $\tilde \phi$ is not increasing, as shown in the implication \ref{it:evi_int} $\Rightarrow$ \ref{it:evi_diff}.
\end{enumerate}
We now proceed with the detailed proofs of the various implications.

  \ref{it:evi_diff} $\Rightarrow$ \ref{it:evi_int}:  Fix $s>0$ and let $\zeta(t) : =c(x,x_t)$ and $\eta(t) : =\frac{\lambda}{2}c(x,x_t)+\phi(x_t)-\phi(x)$. Both functions are lower semicontinuous, since $\phi$, $c(x,\cdot)$, and $\lambda c(x,\cdot)$ are such, and $x(\cdot)$ is $\sigma$-continuous.  Then, \eqref{eq:evi_diff} and \cite[Lemma A.1]{MuratoriSavare} gives that $\zeta,\eta\in L^1_{ \rm loc}( \R_{>0} )$ and $t\mapsto \zeta(t)+\int_s^t \eta(r) \de  r $ is non-increasing on $[s,+\infty)$ which is precisely \eqref{eq:evi_int}. The lower semicontinuity, whence the measurability of $(s,t) \mapsto c(x_s, x_t)$ follows from the $\sigma$-continuity of $x(\cdot)$ and  from  the  joint  $\sigma$-lower semicontinuity of $c$.

 \ref{it:evi_exp_int} $\Rightarrow$ \ref{it:evi_diff}:   Inequality  \eqref{eq:evi_exp_int}  in particular implies  that $\phi(x_t)<+\infty$, so  that  we can choose $x=x_s\in\dom{\phi}$ in \eqref{eq:evi_exp_int}.  Since $c(x_s,x_s)=0$  by \eqref{eq:addass},  this   gives 
\begin{equation}\label{eq:evi_int_applied}
    0\le\frac{  e^{\lambda(t-s)}}{E_\lambda(t-s)}c(x_s,x_t) \le \phi(x_s)-\phi(x_t).
\end{equation}
 This proves that  $\tilde \phi$ is non-increasing,  hence  also locally bounded  in $ \R_{>0}$. As $E_\lambda$ is continuous and $E_\lambda(0)=0$, we can take limits in \eqref{eq:evi_exp_int}, obtaining for all $t_0>0$ and $x \in \dom{\phi}$
\begin{equation*}
    \limsup_{t\downarrow t_0} c(x,x_t) \le c(x,x_{t_0})\le \liminf_{s\uparrow t_0} c(x,x_s).
\end{equation*}
 The second inequality entails that $t\mapsto c(x,x_t)$ is left $\sigma$-l.s.c.  Evaluating at $x=x_{t_0}$  gives $\limsup_{t\downarrow t_0} c(x_{t_0},x_t)=0$,  hence that $\lim_{t\downarrow t_0} c(x_{t_0},x_t)=0$. The compatibility  with $\sigma$ in \Cref{def:comp}-(1) gives that $x(\cdot)$ is right $\sigma$-continuous, whence $t\mapsto c(x,x_t)$ is also right l.s.c.,  hence l.s.c. Since $\tilde \phi$ is non-increasing, $x(\cdot)$ is right $\sigma$-continuous and $\phi$ is $\sigma$-l.s.c.,  we have that  $\tilde \phi$ is l.s.c.   As the   $\limsup$ of a sum is larger than  the  sum of  a  $\limsup$ and  a  $\liminf$, we  moreover have   that  $$\frac{\de^+}{\de t} \bigr |_{t=t_0} (e^{\lambda(t-t_0)}c(x,x_t))\ge\frac{\de^+}{\de t} \bigr |_{t=t_0}c(x,x_{t})+\liminf_{h\downarrow 0}\lambda c(x,x_{t_0+h})\quad \forall t_0>0.$$ 
  Note that we have $\liminf_{h\downarrow 0}\lambda c(x,x_{t_0+h})\ge \lambda c(x,x_{t_0})$ both for $\lambda \geq 0$, as $c$ is  jointly  $\sigma$-l.s.c., and $\lambda<0$, as $c$ is assumed to be separately $\sigma$-continuous in this case. By 
  dividing by $(t-s)$  in \eqref{eq:evi_exp_int} and taking the limit $t\downarrow s$, we  hence  obtain \eqref{eq:evi_diff}. 

 To conclude, we have to show that $x(\cdot)$ is left $\sigma$-continuous.  Letting   $t>0$, we use the non-increasingness of $\tilde \phi$ and \eqref{eq:evi_int_applied} to obtain for $s\in[t/2,t]$
 \begin{equation}\label{eq:evi_int_applied2}
    0\le c(x_s,x_t) \le \frac{E_\lambda(t-s)}{  e^{\lambda(t-s)}}(\phi(x_s)-\phi(x_{t})) \le \frac{E_\lambda(t-s)}{  e^{\lambda(t-s)}}(\phi(x_{t/2})-\phi(x_{t})).
\end{equation}
Taking the  $\limsup$ as $s\uparrow t$, we  deduce  $\lim_{s \uparrow t} c(x_s,x_t)=0$ so that, by the compatibility with $\sigma$ in \Cref{def:comp}-(1), we get that  $x_s \stackrel{\sigma}{\to} x_t$ and   $x(\cdot)$ is left $\sigma$-continuous at $t$.

 \ref{it:evi_int} $\Rightarrow$ \ref{it:evi_diff}:   First of all, by taking limits in \eqref{eq:evi_int}, for all $t_0>0$ and $x \in \dom{\phi}$, we have  that  $ \limsup_{t\downarrow t_0} c(x,x_t) \le c(x,x_{t_0})$.  Reasoning as in the previous implication, we deduce that $x(\cdot)$ is right $\sigma$-continuous. 
 Now we show that $\tilde{\phi}$ is non-increasing, first when restricted to its Lebesgue points, and then in general.

 Let $0<t_0<t_1$ be Lebesgue points for $\tilde \phi$.  For  a.e.~$s \in [t_0,t_1]$ we  write  \eqref{eq:evi_int} for $t=s+h$ and $x=x_s$. Integrating w.r.t.~$s\in [t_0,t_1]$ we obtain 
\begin{align}
    &\int_{t_0}^{t_1} c(x_s,x_{s+h})\de s+\lambda\int_{t_0}^{t_1}\int_{s}^{s+h} c(x_s,x_{ r }) \de r \de s \nonumber\\
    &\quad \le \int_{t_0}^{t_1}\int_{s}^{s+h}\left(\tilde \phi(s)- \tilde \phi(r)\right) \de r \de s \eqqcolon h^2\eta(h).    \label{eq:evi_int_applied3}
\end{align}
Now, let us work on the r.h.s.\ to relate it to $\tilde \phi(t_0)-\tilde \phi(t_1)$.  We write 
\begin{align}
    & \int_{t_0}^{t_1}\int_{s}^{s+h}\left(\tilde \phi(s)- \tilde \phi(r)\right) \de r \de s = \int_{t_0}^{t_1}\int_{0}^{h}\left(\tilde \phi(s)- \tilde \phi(s+r)\right) \de r \de s \nonumber \\
    &\quad = \int_{0}^{h} \int_{0}^{r}\left(\tilde \phi(t_0+\xi)- \tilde \phi(t_1+\xi)\right) \de \xi \de r \nonumber\\
    &\quad= \int_{0}^{h} \int_{0}^{h}\left(\tilde \phi(t_0+\xi)- \tilde \phi(t_1+\xi)\right)\chi_{[0,r]}(\xi) \de \xi \de r \nonumber \\
    &\quad= \int_{0}^{h}\left(\tilde \phi(t_0+\xi)- \tilde \phi(t_1+\xi)\right)(h-\xi)\de \xi\nonumber\\
    &\quad= h^2\int_{0}^{1}\left(\tilde \phi(t_0+h\xi)- \tilde \phi(t_1+h\xi)\right)(1-\xi)\de \xi. \label{eq:bound_eta}
\end{align}
 We set $u(h):= \int_{t_0}^{t_1} c(x_s, x_{s+h})\de s$ so that \eqref{eq:evi_int_applied3} yields
\[ u(h)+ \lambda \int_0^h u(r) \de r \le h^2 \eta(h) \quad  \forall  h>0.\]
An application of Gronwall's lemma entails
\begin{equation}\label{eq:afterg}
h^{-2}u(h) \le e^{h  \lambda^-}\sup_{ \delta \in [0,h]} \eta(\delta)    
\end{equation}
 with $ \lambda^-:=\max\{0,-\lambda\}$  
so that, passing to the $\limsup_{h \downarrow 0}$  in \eqref{eq:evi_int_applied3},  we get 
\begin{equation}\label{eq:limsup_bound}
    0 \stackrel{0\le c}{\le }  \limsup_{h\downarrow 0}\int_{t_0}^{t_1} \frac{c(x_s,x_{s+h})}{h^2}\de s \le \lim_{h \downarrow 0} \eta(h) = \frac{1}{2}(\tilde \phi(t_0)-\tilde \phi(t_1)),
\end{equation}
where the last equality follows by, e.g.,~\cite[Theorem 2.1]{realan} and the fact that $t_0$ and $t_1$ are Lebesgue points of $\tilde{\phi}$. 
Hence, $\tilde \phi$ is non-increasing when restricted  to its Lebesgue points.  However, since $\phi$ is l.s.c.and $x(\cdot)$ is right $\sigma$-continuous, for every Lebesgue point $t_0$ and any $t>t_0$  we have  that $\tilde \phi(t_0) \ge \tilde \phi(t)$.  We deduce in particular that $x_t \in \mathscr{D}(\phi)$ for every $t>0$ and that, if $s,t$ and $h>0$ are given  in such a way that   that $0<s<s+h<t$, then $\tilde \phi (t) \le \tilde \phi (t_0)$ for every Lebesgue point $t_0$ of $\tilde \phi$ in $(s,s+h)$.  This in particular, gives that   $\tilde \phi(t) \le \frac{1}{h} \int_0^h \tilde \phi(s+r) \de r$. 

 We now fix $0<s<t$ and  take any $h>0$ such that $s+h<t$.  By applying  \eqref{eq:evi_int} to $x=x_s$  and $t=s+h$,  discarding  the non-negative term $c(x_s,x_{s+h})$, and  using the above   bound  for $\tilde{\phi}(t)$  we  get 
\begin{equation*}
    \tilde{\phi}(t) - |\lambda| \int_0^1 c(x_s,x_{s+rh})\de r \le \tilde \phi(s) \quad \forall \,  0<s<s+h<t .
\end{equation*}
 Passing to $\limsup_{h \downarrow 0}$  and using the  joint   $\sigma$-lower semicontinuity of $c$ and the $\sigma$-right continuity of $x(\cdot)$, an application of Fatou's lemma (recall  that  $c$ is bounded from below) gives $\tilde \phi(t) \le \tilde \phi (s)$ for every $0 < s< t$, as desired.

 Using that $\tilde{\phi}$ is non-increasing, we show next that $x(\cdot)$ is left  $\sigma$-continuous.  Choosing  $x=x_s$ in  \eqref{eq:evi_int}  and setting   $v(t):=c(x_s,x_t)$ we obtain
\begin{equation}\label{eq:nowgron}
    v(t) \le |\lambda| \int_s^t v(r) \de r + \int_s^t (\tilde{\phi}(s)-\tilde{\phi}(r)) \de r \quad \forall \, 0< s \le t.
\end{equation}
Now fix any $t>0$ and take any $0<s_0<t<t_0$.  By applying Gronwall's lemma to \eqref{eq:nowgron} and integrating by parts we deduce 
\begin{align*}
v(t) &\le \int_s^t (\tilde{\phi}(s)-\tilde{\phi}(r))\de r + \int_s^t \int_s^r(\tilde{\phi}(s)-\tilde{\phi}(u))\de u \,|\lambda | e^{|\lambda|(t-r)} \de r \\
& \le \int_s^t (\tilde{\phi}(s_0)-\tilde{\phi}(t_0))\de r + \int_s^t \int_s^r(\tilde{\phi}(s_0)-\tilde{\phi}(t_0))\de u \,|\lambda | e^{|\lambda| (t-r)} \de r \\
& = (\tilde{\phi}(s_0)-\tilde{\phi}(t_0))\int_0^{t-s} e^{|\lambda| r} \de r 
\end{align*}
for every $s_0<s<t$.  Taking the $\limsup_{s \uparrow t}$ in  the above inequality we deduce that $\limsup_{s \uparrow t} c(x_s,x_t)\le 0$, thus giving, via the usual compatibility conditions, that $x(\cdot)$ is left $\sigma$-continuous.

Furthermore, since $\tilde \phi$ is l.s.c., we have 
\begin{equation}\label{eq:phi_lsc_proof}
    \phi(x_s) \le \liminf_{h\downarrow 0} \tilde \phi(s+h)\le \liminf_{h\downarrow 0} \frac{1}{h} \int_s^{s+h}  \tilde \phi(r) \de r.
\end{equation}
So, when dividing \eqref{eq:evi_int} by $(t-s)$ and taking the $\limsup$ for $t\downarrow s$, the r.h.s.\ can be bounded from above by $\phi(x)-\phi(x_s)$ and we obtain \eqref{eq:evi_diff} using Fatou's lemma for $\lambda \ge 0$ and the   separate   $\sigma$-continuity of $c$ and  the dominated convergence theorem for $\lambda<0$. 

 \ref{it:evi_diff} $\Rightarrow$ \ref{it:evi_exp_int}:  Multiplying \eqref{eq:evi_diff} by $e^{\lambda t}$ gives
\begin{align*}
     &\frac{\de^+}{\de t}(e^{\lambda t} c(x,x_t)) = \limsup_{h \downarrow 0 } \frac{e^{\lambda(t+h)}c(x,x_{t+h})-e^{\lambda t} c(x, x_t)}{h} \\
    &\quad= \limsup_{h \downarrow 0 } \left [e^{\lambda(t+h)} \frac{c(x,x_{t+h})-c(x,x_t)}{h} + c(x,x_t) \frac{e^{\lambda(t+h)}-e^{\lambda t}}{h} \right ] \\
& \quad\le \limsup_{h \downarrow 0 } \left [e^{\lambda(t+h)} \frac{c(x,x_{t+h})-c(x,x_t)}{h} \right ] + \limsup_{h \downarrow 0} \left [ c(x,x_t) \frac{e^{\lambda(t+h)}-e^{\lambda t}}{h} \right ] \\
& \quad= e^{\lambda t} \frac{\de^+}{\de t} c(x,x_t) + \lambda e^{\lambda t} c(x,x_t) \le e^{\lambda t} (\phi(x)-\phi(x_t)).
\end{align*}
Integrating the latter through \cite[Lemma A.1]{MuratoriSavare} and  using the fact that $\tilde \phi$  does not increase,  which is  shown in \ref{it:evi_int} $\Leftrightarrow$ \ref{it:evi_diff},  this  proves  that \eqref{eq:evi_exp_int} holds. 
\end{proof}

\begin{definition}[EVI solution]
\label{def:evisolution}  Let $\phi:X\to (-\infty, + \infty]$,  and   $c$  satisfy  \eqref{eq:addass}. A trajectory  $x:\R_{>0} \to X$ satisfying \eqref{eq:evi_exp_int} in  Theorem \ref{thm:equiv} for some $\lambda \in \R$ is called  a \emph{EVI solution} for the triplet $(X,c, \phi)$  and we  write $x(\cdot) \in \EVI_\lambda(X,c, \phi)$.  Equivalently, $x(\cdot) \in \EVI_\lambda(X,c, \phi)$ if it satisfies either \eqref{eq:evi_diff} or \eqref{eq:evi_int} for a topology $\sigma$ on $X$ which is compatible with the pair $(c, \phi)$ (with $c$ additionally separately $\sigma$-continuous, if $\lambda<0$). 
\end{definition}

\begin{remark}[Definition up to time  $t=0$]\label{rem:t0}  The equivalences of  Theorem \ref{thm:equiv} holds also for a curve $x:[0,+\infty) \to X$ down to time $t=0$, provided  that  the cost $c$ is  additionally  separately $\sigma$-continuous (and not just  jointly  $\sigma$-lower semicontinuous).   Indeed,  in this case  relation  \ref{it:evi_int} written for $0<s \le t$ is equivalent to \ref{it:evi_int} written  for  $0 \le s \le t$ which is in turn equivalent to \ref{it:evi_exp_int}  for  $0 \le s \le t$. To conclude, it is enough to observe that the implication \ref{it:evi_exp_int} $\Rightarrow$ \ref{it:evi_diff} works exactly  as above   upon replacing $t_0>0$ with $t_0 \ge 0$.

 On the basis of these observations, if $c$ is  additionally  separately $\sigma$-continuous,  one can extend the definition of EVI solutions to curves $x:[0,+\infty) \to X$ and, without introducing new notation, use the same symbol $\EVI_\lambda(X,c, \phi)$.  In particular, notice that every curve $(x_t)_{t \ge 0} \in \EVI_\lambda(X,c, \phi)$ is $\sigma$-continuous in $[0,+\infty)$.
\end{remark}

We now introduce two concepts which are the analogues in the case $c={d^2}/{2}$ to the metric derivative and to an oriented local slope.

\begin{definition}[$c$-cost derivative, oriented local slope] Let $\phi:X\to (-\infty, + \infty]$, $c$ be a symmetric cost satisfying \eqref{eq:addass}, and $x:[0,+\infty)\to X$. We  say that $x(\cdot)$ has a \emph{$c$-cost derivative} at time $t  \ge 0 $ if
    \begin{equation}\label{eq:def_c_derivative}
        |\dot x_{t+ }|_c^2:=\lim_{h\downarrow 0}\frac{2c(x_t,x_{t+h})}{h^2}
    \end{equation}
    exists and is finite. We also define the oriented local slope of $\phi$ along $x(\cdot)$ at time $t\ge 0$ as
\begin{equation}\label{eq:def_local_slope_along}
    |\partial \phi|_c(t;x(\cdot)):=\limsup_{h\downarrow 0} \frac{(\phi(x_t)-\phi(x_{t+h}))^+}{\sqrt{2 c(x_{t},x_{t+h})}}.
   \end{equation}
\end{definition}
Symmetry of $c$ proves to be a crucial requirement: in particular, it allows us to derive the $\lambda$-contraction \eqref{eq:lambda_contraction} and the energy identity \eqref{eq:energy_identity} in the proof of Theorem \ref{thm:eviprop}. Notice that symmetry is however not needed in the definition of an EVI solution. Note also that by using the EVI in the integral form \eqref{eq:evi_int} it is easy to give a lower (resp.\ upper) bound to $c(x_t,x_{t+h})$ (resp.\ $c(x_{t+h},x_t)$). Symmetry then justifies that the $c$-derivative is a limit. 

Concerning the slope, as shown in \eqref{eq:evi_int_applied3}, the EVI formulation carries a $h^2$ scaling which justifies taking a square root of $c$ when considering local quantities such as \eqref{eq:def_c_derivative} and \eqref{eq:def_local_slope_along}. Though the cost $c$ itself does not have to have be to be a square metric, if we assume (local) metric properties, then the full usual energy identity can be recovered, see \Cref{cor:metric_case}.

The following theorem is in the spirit of \cite[Theorem 3.5]{MuratoriSavare}, and its forerunner \cite[Theorem 6.9]{Danieri2014}.

\begin{theorem}[Properties of  EVI solutions]\label{thm:eviprop} Let $\lambda \in \R$, $\phi:X\to (-\infty, + \infty]$,  $c$ be a symmetric cost satisfying \eqref{eq:addass},  $\sigma$ be compatible with the pair $(c,\phi)$, let $x:\R_{>0} \to X$ belong to  $ \EVI_\lambda(X,c,\phi)$, and assume that $c$ is additionally separately $\sigma$-continuous if $\lambda<0$. Then the following holds.
\begin{itemize}[leftmargin=5mm]
    \item  {\bf $\lambda$-contraction. }Let $\tilde x(\cdot)\in \EVI_\lambda(X,c,\phi)$. Then,  
    \begin{equation}
    c(x_t,\tilde x_t) \le e^{-2\lambda (t-s)} c(x_s,\tilde x_s) \quad \forall \, 0 < s \le t. \label{eq:lambda_contraction}
    \end{equation}
    \item {\bf Energy identity and regularizing effect.}  For every $t>0$ the following right limits exist, are finite, and are equal:
    \begin{equation}\label{eq:energy_identity}
    \frac{\de}{\de t}\phi(x_{t+})= -  |\dot x_{t+ }|_c^2.   \end{equation}
     Moreover the map $t\mapsto e^{2\lambda t}|\dot x_{t+}|^2_c$ is non-increasing and the map $t \mapsto \phi(x_t)$ is continuous and (locally semi-, if $\lambda <0$) convex in  $\R_{>0}$.  As a consequence, $t \mapsto \phi(x_t)$ is differentiable outside a countable set $\mathcal{T} \subset  \R_{>0}$. Moreover it holds
     \begin{equation}\label{eq:locorsl}
            |\dot x_{t+ }|_c = |\partial \phi|_c(t;x(\cdot)) \quad \forall \, t>0.
     \end{equation}
    \item {\bf A priori estimates}. For every $x\in\dom{\phi}$, we have the following refinement of \eqref{eq:evi_exp_int}:  
    \begin{align}
    e^{\lambda(t-s)}c(x,x_t)-c(x,x_s) + \frac{(E_\lambda(t-s))^2}{2}|\dot x_{t+ }|_c^2 
    \le E_\lambda(t-s)\left(\phi(x)-\phi(x_t)\right) \quad \forall \, 0 < s \le t.    \label{eq:evi_exp_int_ref}
    \end{align}
    
    \item {\bf Asymptotic behaviour as $t\to +\infty$.} Assume that $\phi$ is bounded from below. If $\lambda>0$ and either {\it i)}
$\sigma$ is Cauchy-compatible with $(c,\phi)$ or {\it ii)} the sublevel sets $\phi$ are  $\sigma$-compact,  then $\phi$ has a unique minimum point $\x$ and
    \begin{align}
     &\lambda c(\x,x_t)\le \phi(x_t)-\phi(\x) \le \frac{\lambda c(\x,x_{t_0})}{e^{\lambda (t-t_0)}-1} \text{ and} \nonumber\\
    &c(\x,x_t)\le c(\x,x_{t_0})e^{- 2 \lambda(t-t_0)} \quad  \forall \, 0 < t_0  <  t. \label{eq:asympt_lambda}
    \end{align}
    If $\lambda=0$ and there exists a minimizer $\x$ of $\phi$, then $t \mapsto c(\x,x_t)$ is non-increasing and
    \begin{equation}
        \phi(x_t)-\phi(\x) \le \frac{c(\x,x_{t_0})}{t-t_0} \quad  \forall \, 0 < t_0 < t. \label{eq:asympt_lambda0}
    \end{equation}
   For $\lambda\ge0$  and $\x$ a minimizer of $\phi$, we  furthermore have,  for every $0<t_0 < t$, that
    \begin{equation}
    2\lambda(\phi(x_t)-\phi(\x)) \le |\dot x_{t+ }|_c^2, \quad |\dot x_{t+ }|_c \le |\dot x_{t_0+ }|_c e^{-\lambda(t-t_0)}, \quad  |\dot x_{t+ }|_c\le \frac{ \sqrt{2c(\x,x_{t_0})}}{E_\lambda(t-t_0)}.\label{eq:asympt_lambda_deriv}
    \end{equation}
     If $\lambda=0$, $\phi$ has  $\sigma$-compact sublevel sets,  and $c$ is  jointly  $\sigma$-continuous,  then $x_t$ $\sigma$-converges to a minimizer of $\phi$  as  $t\to +\infty$.
\end{itemize}

\end{theorem}

    \begin{remark}[Time $t=0$] Following up on Remark \ref{rem:t0}, we  briefly discuss the possibility of extending  
    some of the results of Theorem \ref{thm:eviprop}  to the case of  an EVI  solution  $x:[0,+\infty) \to X$ 
    defined  at time $t=0$,  as well.  Recall that, in this case, the cost function $c$ is assumed to be  additionally  separately $\sigma$-continuous in order for the three formulations in Theorem \ref{thm:equiv} (written down to time $t=0$) to be equivalent.  We have the following: 
    \begin{enumerate}
    \item  ($\lambda$-contractivity and uniqueness)  If $\tilde{x}: [0,+\infty) \to X$ belongs to $\EVI_\lambda(X,c,\phi)$ and $c$ is jointly $\sigma$-continuous at $(x_0, \tilde{x}_0)$,  the contractivity estimate  \eqref{eq:lambda_contraction} holds also for $0 \le s \le t$. In particular, if $x_0\in \dom{\phi}$ and $c$ is jointly $\sigma$-continuous at $(x_0, x_0)$, there exists at most one EVI solution starting from $x_0$.
    \item (Stability w.r.t.~the initial condition) If $x^n:[0,+\infty) \to X$ belong to $\EVI_\lambda(X,c,\phi)$, $ x^n_0\stackrel{\sigma}{\to}x_0$,  and $c$ is jointly $\sigma$-continuous at the points $(x_0^n, x_0)$ for every $n \in \N$,  then  $x^n_t \stackrel{\sigma}{\to}x_t$ for all  $t \ge 0$. This  directly  follows from \eqref{eq:lambda_contraction} written  for $x^n(\cdot)$ and $x(\cdot)$  at times $0=s\le t$.
        
    \end{enumerate}
        
    \end{remark}

\begin{proof}[Proof of Theorem \ref{thm:eviprop}] 
    \tb{\bf  $\lambda$-contraction.}  Fix  $0 < s \le t$. We write \eqref{eq:evi_exp_int} for  $x_t$ with test point $ x=\tilde{x}_t$, we multiply it by $\rme^{\lambda(t-s)}$ and, using the symmetry of $c$, we sum it to \eqref{eq:evi_exp_int} written for $\tilde{x}_t$  with test point $ x= x_s$. The terms $e^{\lambda(t-s)}c(\tilde{x}_t, x_s)$ compensate by symmetry and we obtain
    \begin{align*}
    &\rme^{2\lambda(t-s)} c(\tilde{x}_t, x_t) - c(x_s, \tilde{x}_s) \\
    &\quad \le \mathrm{E}_\lambda(t-s) \left (\phi(x_s)-\phi(x_t) \right ) + \mathrm{E}_\lambda(t-s) \left ( e^{\lambda(t-s)}-1 \right ) \left ( \phi(\tilde{x}_t) - \phi(x_t) \right )
    \end{align*}
    for every $0  <  s \le t < + \infty$. Reversing the roles of  $x(\cdot)$ and $\tilde{x}(\cdot)$,  summing up and multiplying by $\rme^{2\lambda s}$ we get
   \[ 2 \rme^{2\lambda t} c(x_t, \tilde{x}_t) - 2\rme^{2\lambda s} c(x_s, \tilde{x}_s) \le \rme^{2\lambda s} \mathrm{E}_\lambda(t-s) \left ( \phi(x_s)-\phi(x_t) + \phi(\tilde{x}_s)-\phi(\tilde{x}_t) \right )\]
    for every $0 \le s \le t < + \infty$. We  now  fix $s>0$,  divide by $t-s>0$,  and  pass to the $\limsup$ as $t \downarrow s$. Observing that $\lim_{r \downarrow 0}\mathrm{E}_\lambda(r)/r=1$ and using that $\tilde \phi$ is $\sigma$-l.s.c. so that the r.h.s.\ is nonpositive in the limit, we obtain
    \[ \frac{\de^+}{\de s} \rme^{2\lambda s} c(x_s, \tilde{x}_s) \le 0 \quad \text{ for every } s >0.\]
    An application of \cite[Lemma A.1]{MuratoriSavare} gives  \eqref{eq:lambda_contraction}.\\

    \tb{\bf Energy identity and Regularizing effect.} We proceed to prove the energy identity, which will be achieved through the following roadmap given by, for every $t>0$, 
    \begin{gather*}
         \limsup_{h\downarrow 0}\frac{2c(x_{t},x_{t+h})}{h^2} \stackrel{\eqref{eq:bound_limsup_lowerDini}}{\le}  -\frac{\de}{\de t^+}\phi(x_{t}) \stackrel{\eqref{eq:energy_bound_local2}}{\le} \limsup_{h\downarrow 0}\frac{2c(x_{t},x_{t+h})}{h^2}, \\
         -\frac{\de^+}{\de t}\phi(x_{t})+\frac{1}{2}\frac{\de}{\de t^+}\phi(x_{t}) \stackrel{\eqref{eq:energy_bound_local_liminf}}{\le}  \liminf_{h\downarrow 0}\frac{c(x_{t},x_{t+h})}{h^2},\\
         \frac{\de^+}{\de t}\phi(x_{t}) = \frac{\de}{\de t^+}\phi(x_{t}).
    \end{gather*}
The results \eqref{eq:energy_bound_local2}, \eqref{eq:bound_limsup_lowerDini},  and  \eqref{eq:energy_bound_local_liminf} will be derived directly from the EVI. The first line gives that $\frac{\de}{\de t_0^+}\phi(x_{t_0})$ is equal to $-\limsup_{h\downarrow 0}\frac{2c(x_{t_0},x_{t_0+h})}{h^2}$. The equality between the Dini derivatives is a consequence of the fact that if a Dini derivative of a continuous
function is non-decreasing, then the function is convex, hence left- and right-differentiable at each point in the interior of its domain, see, e.g.,~\cite[Lemma 1]{BruOstr62}. So we will also need to show the continuity of $\tilde{\phi}=\phi \circ x$.\\

\textit{Finiteness of $\limsup_{h\downarrow 0}\frac{2c(x_{t},x_{t+h})}{h^2}$.} We start from \eqref{eq:afterg} in the  proof of the implication \ref{it:evi_int} $\Rightarrow$ \ref{it:evi_diff} in Theorem \ref{thm:equiv}: recalling the definition of $\eta$ and $u$ in \eqref{eq:evi_int_applied3} and below \eqref{eq:bound_eta}, respectively,  one has  that for every $0< t_0 < t_2$ and $h>0$ it holds, for $ \lambda^-:=\max\{0,-\lambda\}$,
    \begin{align*}\int_{t_0}^{t_2} \frac{c(x_s,x_{s+h})}{h^2} \de s &\le e^{h \lambda^-} \sup_{0 \le \delta \le h} \int_0^1 \left (\tilde{\phi}(t_0+\delta \xi)-\tilde{\phi}(t_2+\delta \xi) \right)(1-\xi) \de \xi \\
    & \le e^{h \lambda^-} \frac{1}{2}(\tilde{\phi}(t_0)-\tilde{\phi}(t_2+h)), 
    \end{align*}
    where we used that $\tilde{\phi}$ is non-increasing. Note that \eqref{eq:afterg} is written for Lebesgue points of $\tilde{\phi}$ but it holds in general (the Lebesgue property is used only later in \eqref{eq:limsup_bound}). Now let $0<t_0<t_1$ and, for any $t_2 \in (t_0, t_1)$ and $h>0$, we  bound from below the left-hand side of  the above inequality using the contractivity \eqref{eq:lambda_contraction} for the curves $r \mapsto x_r$, $r \mapsto \tilde{x}_r:= x_{r+h}$ at times $t_1>s$ to get
    \begin{equation}\label{eq:tobediv}
       \frac{c(x_{t_1+h}, x_{t_1})}{h^2} \int_{t_0}^{t_2} e^{2\lambda(t_1-s)} \de s  \le \int_{t_0}^{t_2} \frac{c(x_s,x_{s+h})}{h^2}\de s \le \frac{e^{h\lambda^-}}{2}(\tilde{\phi}(t_0)-\tilde{\phi}(t_2+h)).
    \end{equation}
    Hence, taking a $\limsup$ as $h \downarrow 0$  by   keeping $t_0, t_1$ and $t_2$ fixed, we get that 

\begin{equation}\label{eq:isfiniteright}
\limsup_{h\downarrow 0}\frac{2c(x_{t},x_{t+h})}{h^2}<+\infty \quad \text{ for every } t>0.
\end{equation}
We can derive the same inequality for the left derivatives: let $s>0$ and, for every $h \in (0,s/2)$, apply the contractivity \eqref{eq:lambda_contraction} to the curves $r \mapsto x_r$ and $r \mapsto x_{r+h}$ at times $s-h>s/2$ to obtain
\[
c(x_{s-h}, x_s) \le e^{-2\lambda(s/2-h)} c(x_{s/2}, x_{s/2+h}).
\]
Dividing by $h^2$ and taking a $\limsup$ as $h \downarrow 0$, and using \eqref{eq:isfiniteright}, we get
\begin{equation}\label{eq:isfiniteleft}
    \limsup_{h \downarrow 0} \frac{c(x_{s-h}, x_s)}{h^2} \le e^{-\lambda s} \limsup_{h \downarrow 0} \frac{c(x_{s/2}, x_{s/2+h})}{h^2} < + \infty \quad \text{ for every } s>0.
\end{equation}

\textit{Continuity of $\tilde{\phi}$.} Since we already know that $\tilde{\phi}$ is right continuous and lower semicontinuous, to prove the continuity $\tilde{\phi}$, it is enough to show that
\begin{equation}\label{eq:leftcont}
    \limsup_{s \uparrow t} \tilde{\phi}(s) \le \tilde{\phi}(t) \quad \text{ for every } t>0.
\end{equation}
Let $0<r<s<t$ and let us write \eqref{eq:evi_exp_int} for times $r<s$ and $x=x_t$:
\[
\frac{e^{\lambda(s-r)}c(x_t, x_s) - c(x_t, x_r)}{E_\lambda(s-r)} + \tilde{\phi}(s) \le  \tilde{\phi}(t).
\]
Dividing and multiplying the l.h.s by $(t-r)^2$ and taking a $\limsup$ as $s \uparrow t$, we obtain
\[
-\frac{(t-r)^2}{E_\lambda(t-r)} \frac{c(x_t, x_r)}{(t-r)^2} + \limsup_{s \uparrow t}\tilde{\phi}(s) \le\tilde{\phi}(t),
\]
where we have used that $\lim_{s \to t} c(x_t, x_s)=0$ by \eqref{eq:Lip_c}. Passing to the $\liminf$ as $r \uparrow t$ and using \eqref{eq:isfiniteleft} gives \eqref{eq:leftcont}.  This proves that $\tilde \phi=\phi \circ x$ is continuous.  \\

We now  show some  inequalities on the Dini derivatives of $\tilde{\phi}$. We start again from \eqref{eq:tobediv}: taking first a $\limsup$ as $h \downarrow 0$, dividing by $(t_2-t_0)$, and  then taking the   $\liminf$ as $t_2 \downarrow t_0$ we get
\begin{equation}\label{eq:weakbound}
    e^{2\lambda t_1} \limsup_{h \downarrow 0} \frac{c(x_{t_1+h}, x_{t_1})}{h^2} \le - \frac{e^{2\lambda t_0}}{2} \frac{\de^+}{\de t_0} \tilde \phi (t_0) \quad \text{ for every } 0< t_0 < t_1.
\end{equation}
We write \eqref{eq:evi_int} for $x=x_{t+h}$ and times $0<t<t+h$, $h>0$, and we use \eqref{eq:Lip_c} to get
    \begin{align}
       \frac{c(x_{t+h},x_t)}{h^2} &\ge \int_0^1 \frac{\phi(x_{t+hr}) -\phi(x_{t+h})}{h} \de r - |\lambda | \int_0^1 \frac{c(x_{t+h},x_{t+rh})}{h} \de r \\
       & \ge \int_0^1 \frac{\phi(x_{t+hr}) -\phi(x_{t+h})}{h} \de r - |\lambda | (\phi(x_t)-\phi(x_{t+h})) \int_0^1 \frac{e^{-\lambda h(1-r)}-1}{-\lambda h} \de r \\ \label{eq:reveq}
       & = \int_0^1 \frac{\phi(x_{t+hr}) -\phi(x_{t+h})}{h} \de r -\frac{|\lambda |}{2} (\phi(x_t)-\phi(x_{t+h}))(1+o(1))
    \end{align} 
    as $h \downarrow 0$. We can thus pass to the $\limsup$ as $h \downarrow 0$ discarding thus through lower semicontinuity the non-negative term in $|\lambda |$. Using the inequality $\limsup (a_n+b_n) \ge \limsup a_n + \liminf b_n$, we obtain 
    \begin{align}\nonumber
       \limsup_{h \downarrow 0} \frac{ c(x_t,x_{t+h})}{h^2} &\ge \limsup_{h \downarrow 0}\int_0^1 \frac{\phi(x_{t+hr})-\phi(x_t)+\phi(x_t) -\phi(x_{t+h})}{h} \de r \\
       &\ge \limsup_{h \downarrow 0} \frac{\phi(x_t) -\phi(x_{t+h})}{h}+\liminf_{h \downarrow 0}\int_0^1 \frac{\phi(x_{t+hr})-\phi(x_t)}{hr} r\de r\nonumber\\
       &\ge -\frac{\de}{\de t^+}\phi(x_{t})+\frac{1}{2}\frac{\de}{\de t^+}\phi(x_{t})=-\frac{1}{2}\frac{\de}{\de t^+}\phi(x_{t})\nonumber \\
       &\ge -\frac{1}{2}\frac{\de^+}{\de t}\phi(x_{t}),\label{eq:energy_bound_local2} 
    \end{align}
    where, to pass from the second to the third line, we have used the following inequality
\begin{equation}\label{eq:pierre}
\liminf_{h \downarrow 0}\int_0^1 \frac{\phi(x_{t+hr})-\phi(x_t)}{hr} r\de r  \ge \frac{1}{2} \frac{\de}{\de t^+}\phi(x_{t}).
\end{equation}
To prove  \eqref{eq:pierre}, we cannot use Fatou's lemma since the integrand is nonpositive. Instead let us denote the  r.h.s.\  by $\frac{1}{2} \kappa \in [-\infty, 0]$.  If   $\kappa = -\infty$, there is nothing to prove, so suppose $\kappa >-\infty$.  For every $\eps>0$, by  the very  definition of $\liminf$ in $\frac{\de}{\de t^+}$, we  find  $\delta_\eps>0$ such that
\[ \frac{\phi(x_{t+h'})-\phi(x_t)}{h'} \ge \kappa -\eps \quad \text{ for every }  h' \in (0, \delta_\eps).\]
In particular
\[ \frac{\phi(x_{t+hr})-\phi(x_t)}{hr} r \ge r(\kappa -\eps) \quad \text{ for every } h \in (0, \delta_\eps), \, r \in (0,1),\]
so that
\[ \int_0^1 \frac{\phi(x_{t+hr})-\phi(x_t)}{hr} r\de r  \ge \frac{1}{2} (\kappa - \eps) \quad \text{ for every } h \in (0, \delta_\eps).\]
Passing to the $\liminf$ as $h \downarrow 0$ and then to the limit as $\eps \downarrow 0$, gives \eqref{eq:pierre}. Multiplying \eqref{eq:energy_bound_local2} by $e^{2\lambda t}$ and using \eqref{eq:weakbound} we get
\begin{equation}
    -\frac{e^{2\lambda t_0}}{2} \frac{\de^+}{\de t_0} \tilde{\phi}(t_0) \stackrel{\eqref{eq:weakbound}}{\ge} e^{2\lambda t_1} \limsup_{h \downarrow 0} \frac{c(x_{t_1+h}, x_{t_1})}{h^2} \stackrel{\eqref{eq:energy_bound_local2}}{\ge}  -\frac{e^{2\lambda t_1}}{2} \frac{\de^+}{\de t_1} \tilde{\phi}(t_1) \quad \text{ for every } 0<t_0< t_1.
\end{equation}
This shows that the map
\[
t \mapsto e^{2\lambda t} \frac{\de^+}{\de t} \tilde{\phi}(t) \in [ - e^{2\lambda t}\limsup_{h \downarrow 0} \frac{c(x_{t+h}, x_{t})}{h^2}, 0]
\]
is non-decreasing and always taking finite values by \eqref{eq:isfiniteright}. Fix $\delta>0$  and  note that the continuous function
\[
u_\delta(t):= e^{2\lambda t} \tilde{\phi}(t) - 2\lambda \int_\delta^t \tilde{\phi}(s) e^{2\lambda s} \de s, \quad t \in (\delta, +\infty)
\]
satisfies
\[
\frac{\de^+}{\de t} u_\delta(t) = e^{2\lambda t} \frac{\de^+}{\de t} \tilde{\phi}(t) \quad \text{ for every } t \in (\delta, +\infty).
\]
By ~\cite[Lemma 1]{BruOstr62}  we get that $u_\delta$ is convex in $(\delta, +\infty)$ so that it is right-differentiable in $(\delta, +\infty)$ by the same lemma.  In particular, $\tilde{\phi}$ is right-differentiable in $( 0 , +\infty)$ since $\delta>0$ was arbitrary. This also show that $\tilde{\phi}$ is (locally semi-, if $\lambda <0$) convex in  $\R_{>0}$ since $u_\delta$ is convex.

\medskip
Now let $t_0, h>0$ and write \eqref{eq:evi_int} with $s=t_0 < t_0+h =t$ and $x=x_{t_0}$ to get
\begin{equation}\label{eq:lowerder}
    \frac{c(x_{t_0},x_{t_0+h})}{h^2} + \frac{\lambda}{h^2} \int_{t_0}^{t_0+h} c(x_{t_0}, x_r) \de r
    \le \frac{1}{h^2} \int_{t_0}^{t_0+h} \left (\tilde{\phi}(t_0)-\tilde{\phi}(r) \right) \de r = \int_0^1 \frac{\phi(x_{t_0})-\phi(x_{t_0+hr})}{hr} r\de r.
    \end{equation}
Observe that
\[
\int_{t_0}^{t_0+h} \frac{c(x_{t_0}, x_{r})}{h^2} \de r
= \int_{0}^1 \frac{c(x_{t_0}, x_{t_0+sh})}{h} \de s.
\]
Thanks to \eqref{eq:tobediv}, the integrand is uniformly bounded  independently of $h$.  An application of the reversed Fatou's lemma thus gives
\begin{equation}\label{eq:thelastplease} 
     \limsup_{h \downarrow 0} \int_{t_0}^{t_0+h} \frac{c(x_{t_0}, x_{r})}{h^2} \de r \le \int_0^1 \limsup_{h \downarrow 0} \frac{c(x_{t_0}, x_{t_0+sh})}{h} \de s =0,
\end{equation}
where we  also  used \eqref{eq:isfiniteright} to conclude that the limit inside the integral is $0$. Combining \eqref{eq:thelastplease} with the non-negativity of $c$ then gives
\begin{equation*}
    0=\lim_{h \downarrow 0} \int_{t_0}^{t_0+h} \frac{c(x_{t_0}, x_{r})}{h^2} \de r.
\end{equation*}
Using \eqref{eq:thelastplease} and applying the same argument as in \eqref{eq:pierre} to the r.h.s.\ of \eqref{eq:lowerder} give, taking a $\limsup$ as $h \downarrow 0$ in \eqref{eq:lowerder}, that
    \begin{equation}\label{eq:bound_limsup_lowerDini}
    \limsup_{h \downarrow 0}\frac{c(x_{t_0},x_{t_0+h})}{h^2}
    \le \limsup_{h \downarrow 0} \int_0^1 \frac{\phi(x_{t_0})-\phi(x_{t_0+hr})}{hr} r\de r \le -\frac{1}{2} \frac{\de}{\de t_0^+}\tilde{\phi}(t_0).
    \end{equation}
However, since $\tilde{\phi}$ has been  already checked  to be right-differentiable, we can replace the lower right Dini derivative with the upper one, thus giving 
    \begin{equation}\label{eq:thefirst}
        -\frac{\de^+}{\de t_0}\phi(x_{t_0}) \ge \limsup_{h\downarrow 0}\frac{2c(x_{t_0},x_{t_0+h})}{h^2} \quad \text{ for every } t_0>0.
    \end{equation}
Adding \eqref{eq:thefirst} to the chain of  inequalities  \eqref{eq:energy_bound_local2} implies that \eqref{eq:thefirst} is actually an equality.  The same computations  that led to  \eqref{eq:energy_bound_local2},  repeated with  $\liminf$ instead of $\limsup$ (and thus using the inequality $\liminf (a_n+b_n) \ge \liminf a_n + \liminf b_n$) yield 
\begin{align}\nonumber
       \liminf_{h \downarrow 0} \frac{ c(x_t,x_{t+h})}{h^2}
       &\ge \liminf_{h \downarrow 0} \frac{\phi(x_t) -\phi(x_{t+h})}{h}+\liminf_{h \downarrow 0}\int_0^1 \frac{\phi(x_{t+hr})-\phi(x_t)}{hr} r\de r\nonumber\\
       &\ge -\frac{\de^+}{\de t}\phi(x_{t})+\frac{1}{2}\frac{\de}{\de t^+}\phi(x_{t})=-\frac{1}{2}\frac{\de^+}{\de t}\phi(x_{t})\label{eq:energy_bound_local_liminf} 
    \end{align}
where the last equality is due to $\tilde{\phi}$ having right-derivatives. This  implies that
$\lim_{h \downarrow 0 } \frac{c(x_t,x_{t+h})}{h^2}$  exists at all $t>0$  and \eqref{eq:energy_identity}  holds. \\

    \tb{\bf Oriented local slope.} We now prove \eqref{eq:locorsl}; fix $t>0$. We prove separately the two inequalities, starting with $|\partial \phi|_c(t;x(\cdot)) \le |\dot x_{t+ }|_c$. If $|\dot x_{t+ }|_c=0$, since $s\mapsto e^{\lambda s}|\dot x_{s+ }|_c$ is non-increasing, the $c$-derivative remains zero  at   times larger than $t$. The energy identity \eqref{eq:energy_identity} and the continuity of $s \mapsto \phi(x_s)$ give that $\phi$ is also constant in $[t,\infty)$; this gives that $|\partial \phi|_c(t;x(\cdot))=0$. Now assume that $|\dot x_{t+ }|_c>0$; in particular, we can assume that $c(x_{t+h}, x_t)>0$ for $h>0$ sufficiently small. We start from \eqref{eq:reveq}: 
        \begin{align*}
       \frac{c(x_{t+h},x_t)}{h^2} \ge \int_0^1 \frac{\phi(x_{t+hr}) -\phi(x_{t+h})}{h} \de r -\frac{|\lambda |}{2} (\phi(x_t)-\phi(x_{t+h}))(1+o(1)).
    \end{align*} 
    We multiply by $2h$ and divide by $\sqrt{2c(x_{t+h},x_t)}$, then take a $\limsup_{h \downarrow 0}$. Using the continuity of $\tilde \phi$ and that $|\dot x_{t+ }|_c$ is a limit, we get that
       \begin{align}\nonumber
       |\dot x_{t+ }|_c&=\limsup_{h \downarrow 0} \frac{ \sqrt{2c(x_{t+h},x_t)}}{h} \\
       &\ge 2\limsup_{h \downarrow 0}\int_0^1 \frac{\phi(x_{t+hr})-\phi(x_t)+\phi(x_t) -\phi(x_{t+h})}{\sqrt{2c(x_{t+h},x_t)}} \de r \nonumber \\
       &\ge 2\limsup_{h \downarrow 0} \frac{\phi(x_t) -\phi(x_{t+h})}{\sqrt{2c(x_{t+h},x_t)}}+2\liminf_{h \downarrow 0}\int_0^1 \frac{\phi(x_{t+hr})-\phi(x_t)}{hr} \frac{h}{\sqrt{2c(x_{t+h},x_t)}}r\de r\nonumber\\
       &\ge 2\limsup_{h \downarrow 0} \frac{\phi(x_t) -\phi(x_{t+h})}{\sqrt{2c(x_{t+h},x_t)}} -2\frac{|\dot x_{t+ }|_c^2}{2}\frac{1}{|\dot x_{t+ }|_c}= 2|\partial \phi|_c(t;x(\cdot))-|\dot x_{t+ }|_c.\label{eq:energy_bound_local2_slope} 
    \end{align}
    Hence,
        \begin{equation*}
        |\partial \phi|_c(t;x(\cdot)) \le |\dot x_{t+ }|_c.
    \end{equation*}

    We now proceed like in \cite[eq.(3.33)--(3.34)]{MuratoriSavare} to show the reverse inequality. We fix $\epsilon>0$ and we take $h$ small enough so that for all $r\in[0,1]$ it holds
    \begin{equation*}
        \frac{\phi(x_{t}) -\phi(x_{t+hr})}{\sqrt{2 c(x_{t},x_{t+hr})}} \le |\partial \phi|_c(t;x(\cdot))+\epsilon, \quad \frac{\sqrt{2 c(x_{t},x_{t+hr})}}{h} \le |\dot x_{t+ }|_c+\epsilon, \quad \sqrt{c(x_t, x_{t+hr})} <  \epsilon
    \end{equation*}
    We write \eqref{eq:evi_int} for $x=x_{t}$ at times $t$ and $t+h$, and we divide by $h^2$. Consequently, for $h$ small enough, we have
    \begin{align*}
       \frac{c(x_{t+h},x_t)}{h^2} &\le \int_0^1 \frac{\phi(x_{t}) -\phi(x_{t+hr})}{h} \de r - \lambda  \int_0^1 \frac{c(x_{t},x_{t+hr})}{h} \de r \\
       & = \int_0^1 \frac{\phi(x_{t}) -\phi(x_{t+hr})}{\sqrt{2 c(x_{t},x_{t+hr})}} \frac{\sqrt{2 c(x_{t},x_{t+hr})}}{h}\de r - \lambda  \int_0^1 \sqrt{c(x_{t},x_{t+hr})}\frac{\sqrt{c(x_{t},x_{t+hr})}}{h} \de r\\
       & \le (|\partial \phi|_c(t;x(\cdot))+\epsilon) (|\dot x_{t+ }|_c+\epsilon) + \lambda^-\epsilon (|\dot x_{t+ }|_c+\epsilon).
    \end{align*} 
    Taking the limit $h\downarrow 0$ and then $\epsilon \downarrow 0$, this results in
    \begin{equation*}
        |\dot x_{t+ }|_c^2 \le |\partial \phi|_c(t;x(\cdot)) |\dot x_{t+ }|_c.
    \end{equation*}
    Consequently the local slope along the curve is equal to the $c$-derivative, i.e.\ $|\partial \phi|_c(t;x(\cdot))=|\dot x_{t+ }|_c$.\\

    \tb{\bf A priori estimates.}  To show \eqref{eq:evi_exp_int_ref}  let us start by  noting   that, by taking derivatives in $\tau$ and using the definition of $E_\lambda$, for $\tau \ge 0$ we have
\begin{equation*}
    \frac{(E_{\lambda}(\tau))^2}{2}e^{-2\lambda \tau}=\frac{(E_{-\lambda}(\tau))^2}{2}=\frac{1}{2} \left(\int_0^\tau e^{-\lambda r} \de r\right)^2= \int_0^\tau E_{-\lambda}(r) e^{-\lambda r} \de r.
\end{equation*}
     Hence, using the  fact that $t\mapsto e^{\lambda t}|\dot x_{t+ }|_c$  is non-increasing on $\R_{>0}$,  the   energy identity  \eqref{eq:energy_identity},   the almost everywhere differentiability of  $t \mapsto \phi(x_t)$, and  by  integrating by parts, we obtain
    \begin{align*}
        \frac{(E_\lambda(t-s))^2}{2}|\dot x_{t+ }|_c^2
         &\le \int_0^{t-s} E_{-\lambda}(r) e^{-\lambda r} e^{2\lambda r}|\dot x_{(s+r)+ }|_c^2 \de r \\
         & \hspace{-2.2cm} \stackrel{\eqref{eq:energy_identity}}{ = }  -\int_0^{t-s} E_{-\lambda}(r) e^{\lambda r}\left( \frac{\de}{\de r}\phi(x_{(s+r)})\right) \de r\\
        &\hspace{-2cm}=\int_0^{t-s}  e^{\lambda r} \phi(x_{(s+r)}) \de r- E_{-\lambda}(t-s) e^{\lambda (t-s)}\phi(x_t)\\
        &\hspace{-2cm}=\int_s^{t} e^{\lambda (r-s)}\left(\phi(x_r)-\phi(x_t)\right) \de r \\
         &\hspace{-2.15cm} \stackrel{\eqref{eq:evi_diff}}{\le }  -\int_s^t \left(e^{\lambda (r-s)} c(x_r,x)\right )' \de r + \int_s^t \left(e^{\lambda (r-s)}(\phi(x)-\phi(x_t) \right ) \de r \\
        &\hspace{-2cm} = {}-e^{\lambda(t-s)}c(x,x_t)+c(x,x_s) + E_\lambda(t-s)\left(\phi(x)-\phi(x_t)\right).
    \end{align*}\\
    
    \tb{\bf Asymptotic behaviour as $t\to +\infty$.} Since $t\mapsto \phi(x_t)$  is non-increasing and bounded from below, it converges to some $A\in\R$  as $t \to + \infty$ and   $\{x_t\}_{t >0}$  is contained in a sublevel set of $\phi$.  In the case  $\lambda>0$,  taking  $0<s\le t$ \eqref{eq:evi_exp_int} reads 
    \begin{equation}\label{eq:cv_infty2}
        0 \le c(x,x_t)\le e^{-\lambda (t-s)}E_\lambda(t-s)\big(\phi(x)-\phi(x_t)
        \big)+e^{-\lambda (t-s)}c(x,x_s)  \quad \forall  x \in \dom{\phi}.
    \end{equation}
    Hence, for $x=x_s$ we obtain 
    \begin{equation}\label{eq:cv_infty}
        0 \le c(x_s,x_t)\le \frac{e^{\lambda(t-s)}-1}{\lambda e^{\lambda(t-s)}}(\phi(x_s)-\phi(x_t))\le \frac{|\phi(x_s)-\phi(x_t)|}{\lambda} . 
    \end{equation}
      Since $c$ is symmetric, the inequality holds for any  $t,s \ge 0$.  If i) $\sigma$ is Cauchy-compatible, as $\phi(x_s)$ converges, it is Cauchy, thus by \eqref{eq:cv_infty} $x(\cdot)$ converges to some $\x \in X$  as  $t\to +\infty$. 
    Otherwise,  let us assume that  ii) the sublevel sets of $\phi$ are  $\sigma$-compact. Fix $t_0>0$ and set $A_{t_0}:=\{x\, | \, \phi(x)\le \phi(x_{t_0})\}$. It is enough to show that, given two subnets $(x_{t_\eta})_\eta$ and $(x_{t_\gamma})_\gamma$ converging to points $\x,\x' \in A_{t_0}$,  we necessarily have that  $\x=\x'$. By \eqref{eq:cv_infty}, we have
    \[ 0 \le c(x_{t_\eta}, x_{t_\gamma}) \le \frac{|\phi(x_{t_\eta})- \phi( x_{t_\gamma})|}{\lambda} \quad  \forall  \eta, \gamma.\]
    Using the  joint $\sigma$- lower semicontinuity of $c$ and the fact that $\lim_{t \to + \infty}\phi(x_t)=A$, we deduce that $c(\x,\x')=0$ so that $\x=\x'$.

    Let us show that $\x$ is the unique minimizer of $\phi$.  Taking the   $\liminf$ as $t \to + \infty$ in \eqref{eq:cv_infty2} and, using the  $\sigma$- lower semicontinuity properties of $c$ and $\phi$, we get 
    \begin{equation}\label{eq:cv_infty_min}
        0 \le \lambda c(x,\x)\le \phi(x)-\limsup_{t\to+\infty}\phi(x_t)\le \phi(x)-\phi(\x)  \quad  \forall  x \in \dom{\phi},
    \end{equation}
     which shows that $\x$ is a minimizer of $\phi$ and that it is unique by  the nondegeneracy  of $c$  coming from \eqref{eq:addass} .

    Similarly,  for given $0<t_0<t$,  considering \eqref{eq:evi_exp_int} for  $s=t_0$, and $x=\x$, we obtain
    \begin{equation*}
        -c(\x,x_{t_0})\le e^{\lambda(t-t_0)}c(\x,x_t)-c(\x,x_{t_0})
    \le E_\lambda(t-t_0)\left(\phi(\x)-\phi(x_t)\right) \quad  \text{ for every } 0 < t_0 \le t 
    \end{equation*}
    which  implies $\phi(x_t)-\phi(\x) \le \frac{\lambda c(\x,x_{t_0})}{e^{\lambda (t-t_0)}-1}$.  In combination   with \eqref{eq:cv_infty_min} evaluated at $x=x_t$,  the latter  gives the first  inequality in  \eqref{eq:asympt_lambda}.

     Note that \eqref{eq:cv_infty_min} is  nothing but relation  \eqref{eq:evi_diff} written for the  constant trajectory  $t \mapsto \x$,  proving that it indeed  belongs to $\EVI_\lambda(X,c,\phi)$. Therefore, the second part of \eqref{eq:asympt_lambda} is simply  the  contractivity  \eqref{eq:lambda_contraction} for the curves $t \mapsto x_t$ and $t \mapsto \x$ with $s=t_0$.

      In the case  $\lambda=0$, we use the fact that $t\mapsto \phi(x_t)$  is non-increasing and the integral formulation \eqref{eq:evi_int} with $x=\bar{x}$  and $s=t_0$ to obtain that,  for every $0 < t_0 \le t$, it holds 
    \begin{equation}\label{eq:cv_lambda0}
        (t-t_0)(\phi(x_t)-\phi(\x)) \le \int_{t_0}^t (\phi(x_r)-\phi(\x)) \, {\rm d}r  \stackrel{\eqref{eq:evi_int}}{\le} c(\x,x_{t_0})-c(\x,x_{t}) \stackrel{c \geq 0}{\le} c(\x,x_{t_0}). 
    \end{equation}
    This  gives that $t \mapsto c(\bar{x}, x_t)$ is non-increasing and proves \eqref{eq:asympt_lambda0}.

    The second inequality in \eqref{eq:asympt_lambda_deriv} follows by the already proven fact that $t \mapsto e^{2\lambda t} |\dot{x}_{t+}|_c^2$ is non-increasing  on  $\mathbb{R_+}$.  The first one follows from the same fact,  energy identity \eqref{eq:energy_identity}, and has to be proven only for $\lambda>0$:
    \begin{align*}
        \phi(x_t)-\phi(\x)&=-\int_t^{+\infty} \frac{\de}{\de  r }\phi(x_{r}) \de r  = \int_t^{+\infty} e^{-2 \lambda r} e^{2 \lambda r}|\dot x_{r+ }|^2_c\de r \\
        & \le e^{2 \lambda t}|\dot x_{t+ }|^2_c \int_t^{+\infty} e^{-2 \lambda r} \de r = \frac{|\dot x_{t+ }|^2_c}{2\lambda}.
    \end{align*}
     Using  \eqref{eq:evi_exp_int_ref} with  $x=\x$  and $s=t_0$ we  obtain that, for every $0 < t_0 \le t$, it holds 
    \begin{equation}
    e^{\lambda(t-t_0)}c(\x,x_t)+E_\lambda(t-t_0)\left(\phi(x_t)-\phi(\x)\right)+ \frac{(E_\lambda(t-t_0)) ^2}{2}|\dot x_{t+ }|_c^2
    \le c(\x,x_{t_0}).
    \end{equation}
    As the first two terms on the  l.h.s.\  are non-negative, the third inequality in \eqref{eq:asympt_lambda_deriv} eventually ensues.

     Finally, let us assume that $\lambda=0$ and $\phi$ has $\sigma$-compact sublevel sets, so that it has at least one minimizer $\x$. Since  $\{x_t\}_{t>0}$  is contained in a sublevel set of $\phi$, there exists a subnet $(x_{t_\eta})_\eta$ converging to a point $\x'$. By \eqref{eq:asympt_lambda0}, using the lower semicontinuity of $\phi$, we deduce that 
    \[ \phi(\x') \le \liminf_\eta \phi(x_{t_\eta}) \le \liminf_{\eta} \left [ \phi(\x) + \frac{c(\x,x_{t_0})}{t_\eta-t_0} \right ] = \phi(\x),\]
    where $t_0 >0$ is any value such that $t_\eta >t_0$ eventually in $\eta$. Therefore also $\x'$ is a minimum point of $\phi$.  Since $t \mapsto c(\bar{x}', x_t)$ is non-increasing, there exists the limit $\ell:=\lim_{t \to + \infty} c(\bar{x}', x_t) \in [0,+\infty) $.  Since $c$ is  jointly  $\sigma$-continuous  we deduce that $\lim_\eta c(x_{t_\eta}, \x')=0$.  Hence, we necessarily have that  $\lim_{t \to + \infty} c(x_t, \x')=0$.  By compatibility, this implies  that $x_t$ $\sigma$-converges to $\x'$ as $t \to + \infty$.
    \end{proof}

In the following corollary we show that, whenever the compatibility between $(c,\phi)$ and $\sigma$ is  strengthened  by asking a local metric character of $c$ in appropriate $\sigma$-neighbourhoods, then we can improve \eqref{eq:locorsl} adding the equality with a notion of local slope which coincides with the usual one for $c=\frac{d^2}{2}$ and $\sigma$ the topology induced by $d$.
\begin{corollary}\label{cor:metric_case} Under the same assumptions of Theorem \ref{thm:eviprop}, assume additionally that $c$ is separately $\sigma$-continuous and the following local metric character of $c$: for every $x \in X$ and every $\eps>0$, there exists a $\sigma$-open subset $x \in U \subset X$ such that
\begin{equation}\label{eq:def_locally_square_metric}
\sqrt{c(x,x')}-\sqrt{c(x',x'')} \le  \sqrt{c(x,x'')}(1+ \varepsilon ) \quad \forall \, x',x''\in U.
\end{equation}
Then
\[
|\partial \phi|_c(t;x(\cdot)) = \limsup_{x' \stackrel{\sigma}{\to} x_t} \frac{(\phi(x_t)-\phi(x'))^+}{\sqrt{2 c(x_t,x')}} \quad \forall \, t>0.
\]
\end{corollary}
\begin{proof} Fix $t>0$ and denote the right hand side above by $|\partial \phi|_{c, \sigma}(x_{t})$. Since $x(\cdot)$ is $\sigma$-continuous, we have
    \begin{equation}
        |\partial \phi|_c(t;x(\cdot)) \le |\partial \phi|_{c, \sigma}(x_{t}).
    \end{equation}
    We are left to show the opposite inequality.
     We start with \eqref{eq:evi_int} for test point $x'\neq x_t$ and $t<t+h$, for $h>0$, to get
    \begin{multline}
    \frac{(\sqrt{c(x',x_{t+h})}-\sqrt{c(x',x_{t})})\cdot (\sqrt{c(x',x_{t+h})}+\sqrt{c(x',x_{t})})}{h} + \frac{\lambda}{h}\int_{t}^{t+h} c(x',x_r) \de r \\
    \le \phi(x')-\frac{1}{h}\int_{t}^{t+h} \phi(x_r) \de r.    \label{eq:evi_int_sqrt}
    \end{multline}
    We fix $\eps>0$ and we use the local metric character of $c$ to find a $\sigma$-open subset $U\subset X$ as in \eqref{eq:def_locally_square_metric} with $x=x_t$. Since $x(\cdot)$ is $\sigma$-continuous, we have that $x_{t+h} \in U$ for $h>0$ small enough; therefore, for such values of $h$ and for every $x' \in U$, we deduce from \eqref{eq:evi_int_sqrt} that
    
\begin{equation}
    \frac{\frac{1}{h}\int_{t}^{t+h} \phi(x_r) \de r -\phi(x')}{\sqrt{2c(x_t, x')}} \le  (1+\eps) \frac{\sqrt{c(x_t, x_{t+h})}}{h} \frac{\sqrt{c(x',x_{t+h})}+\sqrt{c(x',x_{t})}}{\sqrt{2c(x_t, x')}} - \frac{\lambda}{h} \int_t^{t+h}c(x',x_r) \de r \label{eq:evi_int_sqrt2}.
\end{equation}
Taking first a $\limsup$ as $h \downarrow 0$, using the continuity of $\tilde \phi$ and of $h \mapsto c(x',x_{t+h})$, we get that for every $x' \in U$ it holds
   \[
   \frac{\phi(x_t)-\phi(x')}{\sqrt{2 c(x_{t},x')}} \le (1+\eps) |\dot x_{t+ }|_c +\frac{\lambda^-}{\sqrt{2}} \sqrt{c(x',x_t)}
   \]
   We take the positive part on both sides and, passing to the $\limsup$ as $x' \stackrel{\sigma}{\to} x_{t}$ and then as $\eps \downarrow 0$, we get
    \begin{equation}
        |\partial \phi|_c(x_{t})\le |\dot x_{t+ }|_c,
    \end{equation}
    concluding the proof.
\end{proof}
We remark that, when $c=\frac{d^2}{2}$ for some metric $d$, redoing the same proof as above with $\limsup$ replaced by $\sup$, leads to the equality between global and local slope along the EVI solution, as given in \cite[formula (3.17)]{MuratoriSavare}.

\subsection{On local characterizations}\label{sec:slope_local}

 In this section, we temporarily depart from the study of EVI  solutions and comment on some alternative formulations of local nature. Indeed, in the metric case $c=d^2/2$ one could consider a hierarchy of weaker notions, which go under the name of {\it Energy-Dissipation Equality/Inequality} (EDE and EDI, for short), as well as of {\it subdifferential formulation}, \cite{Ambrosio2012userguide}. In the metric case, one has the 
chain of implications \cite[Section 3.2]{Ambrosio2012userguide}
\begin{equation*}
    EVI \Rightarrow EDE \Rightarrow EDI \Rightarrow \text{subdifferential formulation}.
\end{equation*}
Except for the EVI, which is global,  the other formulations are of local nature.  Reversing these implications requires  a global property, such as some notion of convexity. This will be considered in \Cref{prop:evi_local_2}.
The following is inspired by \cite[Proposition 3.11]{MuratoriSavare}. 

\begin{proposition}[Local characterization]\label{prop:evi_local_1}  Let $\lambda \ge 0$, $\phi:X\to (-\infty, + \infty]$, $c$ be a cost satisfying \eqref{eq:addass}, 
$\sigma$ be compatible with the pair $(c,\phi)$, and let $x:\R_{>0} \to X$ belong to $ \EVI_\lambda(X,c,\phi)$.  For any  $t_0>0$ and any curve $\gamma :[0,+\infty) \to X$ with $\gamma_0=  x_{t_0}$, we have that
    \begin{equation}\label{eq:evi_local}
    [\dot x, \gamma]_{c,t_0}:= \liminf_{s\downarrow 0} \frac{1}{s} \frac{\de^+}{\de t} c(\gamma_s, x_t)_{|t=t_0} \le \phi'(x_{            t_0};\gamma):= \liminf_{s\downarrow 0} \frac{\phi(\gamma_s)-\phi(x_{t_0})}{s}
    \end{equation}
    where the inequality is intended in $[-\infty,+\infty]$.
\end{proposition}

\begin{proof}
    Notice first that we can  restrict ourselves  to the case $\lambda=0$ since $c$ is non-negative and to curves such that $\gamma_s\in\dom{\phi}$ for all $s\in(0,1]$. Hence setting $x=\gamma_s$ in \eqref{eq:evi_diff} and dividing by $s$, we get
    \begin{equation*}
       \frac{1}{s} \frac{\de^+}{\de t} c(\gamma_s, x_t)_{|t=t_0} \le \frac{\phi(\gamma_s)-\phi(x_{t_0})}{s}
    \end{equation*}
    taking the $\liminf$ on both sides, we obtain \eqref{eq:evi_local}.
\end{proof}
\begin{remark}\label{rmk:subdifferential}
     Equation \eqref{eq:evi_local} corresponds to the generalization to the case of a general cost $c$ of the classical subdifferential formulation.  
    Indeed, consider first for simplicity that $X$ is a Hilbert space  with norm $\| \cdot \|$,  let  $c={\|\cdot\|^2}/{2}$,  and  assume that $\phi$ is differentiable.  Then, taking $v\in X$ and $ \gamma_s=x_{t_0}+sv$, we obtain $[\dot x, \gamma]_{c,t_0}=\bracket{-\dot x_{t_0},v} \le \bracket{\nabla \phi(x_{t_0}),v}$. Since this holds for all $v  \in X$,  we get  that $\dot x_{t_0}=-\nabla \phi(x_{t_0})$. 
    
    More generally, assume that $(X,d)$ is a  locally Hilbertian manifold, that $c\in C^2$ with $c = 0$  on the diagonal, and that  $\phi \in C^1$.  Setting  $\dot \gamma_0=v$  for arbitrary $v \in T_{x_{t_0}}M$, as $0=\nabla_2 c(\gamma_0, x_{t_0})$, we get 
    \begin{align}
    &\liminf_{s\downarrow 0} \frac{\phi(\gamma_s)-\phi(x_{t_0})}{s} \ge [\dot x, \gamma]_{c,t_0}=\liminf_{s\downarrow 0} \frac1s \bracket{\nabla_2 c(\gamma_s, x_{t_0}),\dot x_{t_0}}\nonumber\\
    &\quad = \left\langle\lim_{s\downarrow 0}\frac{\nabla_2 c(\gamma_s, x_{t_0})-\nabla_2 c(\gamma_0, x_{t_0})}{s},\dot x_{t_0}\right\rangle=\bracket{\nabla_{1,2} c(x_{t_0}, x_{t_0}) v,\dot x_{t_0}}\label{eq:evi_local2}
    \end{align}
    which is precisely the definition of the elements of the  Fr\'echet  subdifferential $\partial \phi$, hence $$\nabla_{2,1} c(x_{t_0}, x_{t_0})\dot x_{t_0} \in \partial \phi(x_{t_0}).$$

    We thus recover the seminal observation of Kim and McCann \cite[Section 2]{KimMcCann2010} that the ``geometric information of $c$'' is contained in its mixed-Hessian $\nabla_{2,1} c$. Moreover two costs with the same mixed Hessian will induce the same gradient flows, as shown for $\calW_c$ and $\calW_2$ by \cite{rankin2024jkoschemesgeneraltransport}.
\end{remark}

We will now consider the converse to \Cref{prop:evi_local_1}, that is whether the local equation \eqref{eq:evi_local}  implies that  $x(\cdot) \in  \EVI_\lambda(X,c,\phi)$.  To this aim, some global property will be needed.  
 Concerning  $\phi$,  this will be some suitable form of $\lambda$-convexity.  Nonetheless, as the cost $c$ is integrated along trajectories, some form of convexity for $c$ can be useful, as well. The latter is naturally connected to curvature conditions of $c$.  Note however that for geodesic metric spaces the validity of \eqref{eq:decreasing_slope} was shown without using (Alexandrov) curvature, see \cite[Lemma 3.13]{MuratoriSavare}.

More precisely, our goal is to prove that, under some  suitable assumption  on $c$, for some class of curves $\gamma(\cdot)$ such that $\gamma_{ 0 }=x_{t_0}$ and $\gamma_{ 1 }=x$ is left free, we have 
\begin{equation}\label{eq:decreasing_slope}
    [\dot x, \gamma]_{c,t_0}:= \liminf_{s\downarrow 0} \frac{1}{s} \frac{\de^+}{\de t} c(\gamma_s, x_t)_{|t=t_0} \ge \frac{\de^+}{\de t} c(\gamma_{ 1}, x_t)_{|t=t_0}.
\end{equation}
This  would indeed correspond  to \cite[eq.~(3.56), Lemma 3.13]{MuratoriSavare}, which holds for  a complete geodesic space  $(X,d)$, $c={d^2}/{2}$, and with $\gamma$  being  any geodesic.  We hence anticipate that one  is   asked to restrict the class of admissible curves $\gamma(\cdot)$ in the case of general costs $c$, as well.  In preparation for \Cref{prop:evi_local_2}, we recall the following Definition from \cite[Definitions 2.6, 2.7]{leger2024nonnegative}.  

\begin{definition}[NNCC space and variational $c$-segment]\label{def:nncc}
    Let 
    $c\colon X\times \YY\to [-\infty,+\infty]$  be given.  We say that $(X\times \YY,c)$ is a \emph{Non-Negatively Cross-Curved (NNCC)  space} if for every $(x_0,x_1,\y)\in X\times X\times \YY$ such that $c(x_0,\y)$ and $c(x_1,\y)$ are finite, there exists a function  $\gamma:[0,1] \to X $ with $\gamma_0=x$ and $\gamma_1=x_1$  such that 
    \begin{align}  
    c(\gamma_s,\y)-c(\gamma_s,y)&\leq (1-s)[c(\gamma(0),\y)-c(\gamma(0),y)] \nonumber \\
     & \quad + s \,[c(\gamma(1),\y)-c(\gamma(1),y)]
     \quad \forall y \in \YY, s \in (0,1),\label{eq:NNCC-ineq}
    \end{align}
    with the rule $(+\infty)+(-\infty)=+\infty$  for  the r.h.s.\  In this case, we say that the function $\gamma(\cdot)$ is a 
    \emph{variational $c$-segment}. 
\end{definition}

 Relation  \eqref{eq:NNCC-ineq} is  often referred to as the  \emph{NNCC inequality}.  Being required for all $y \in \YY$, it is a quite strong property. In particular, variational $c$-segments may not exist.  
 Correspondingly, situations where variational $c$-segments exist are of  special  
 interest. In the metric setting $c={d^2}/{2}$, NNCC spaces are a   specific  subclass  of complete PC metric spaces (\cite[Proposition 2.26]{leger2024nonnegative}).  In particular,  the Wasserstein space is NNCC \cite[Theorem 3.11]{leger2024nonnegative}. A prime example of non-metric NNCC space is a Banach space equipped with a Bregman divergence, see \cite[Example 2.9]{leger2024nonnegative}.

\begin{proposition}[\eqref{eq:evi_local} implies EVI for NNCC spaces] \label{prop:evi_local_2}    Let $\lambda \ge 0$, $\phi:X\to (-\infty, + \infty]$, $c$ be a cost satisfying \eqref{eq:addass}, and $\sigma$ be compatible with the pair $(c,\phi)$.  
Let $x: \R_{>0}  \to \dom{\phi}$ be a $\sigma$-continuous function. Assume that 
\begin{itemize}
    \item[\rm i)] $(X\times X,c)$ is  a  NNCC space; 
    \item[\rm ii)] for all $x_0,\, x \in X$ 
    there exists a variational c-segment $\gamma:[0,1]\to X$  with  $\gamma_{ 0 }=x_0$ and $\gamma_{ 1}=x$, and 
     $M:[0,1]\to \R \cup \{+\infty\}$ with 
    \begin{equation}\label{eq:convexity_g_c}
        \phi(\gamma_t) \le (1-t)\phi(x_0)+t\phi(x)-\lambda t c(x,x_0)+ M(t) 
    \end{equation}
     such that   $\liminf_{t\downarrow 0}\frac{M(t)}{t}=0$; 
    \item[\rm iii)] \eqref{eq:evi_local} holds along these  variational  $c$-segments.
    \end{itemize}
    Then, $x(\cdot)  \in  \EVI_\lambda(X,c,\phi)$. 
\end{proposition}
 Note that relation  \eqref{eq:convexity_g_c} holds if 
$t\mapsto \phi(\gamma_t) - \lambda t^2 c(x,x_0)$  is convex. Moreover, it follows also if $t\mapsto \phi(\gamma_t) - \lambda c(\gamma_t,x_0)$ is convex and  $\liminf_{t\downarrow 0} \frac{\lambda c(\gamma_t,x_0)}{t}=0$. 
These two  cases  coincide when $(X,d)$ is a geodesic metric space, $c=d^2/2$ and $\gamma$ is a geodesic. The  former  is closer to the  condition on  generalized geodesics of \cite{ags08}, while the  latter  is more inline with  the theory of  NNCC spaces and \eqref{eq:NNCC-ineq} (see also \Cref{rmk:compatibility}).

\begin{proof}[Proof of \Cref{prop:evi_local_2}] Let $t_0>0$, $h>0$, $x\in X$ and take $x_0=\y=x_{t_0}$, $x_1=x$, $y=x_{t_0+h}$ in \eqref{eq:NNCC-ineq}, noticing that, since $\gamma(\cdot)$ does not depend on $y$,
    \begin{equation*}
        c(\gamma_s,x_{t_0})-c(\gamma_s,x_{t_0+h})\leq (1-s)[c(x_{t_0},x_{t_0})-c(x_{t_0},x_{t_0+h})] + s \,[c(x,x_{t_0})-c(x,x_{t_0+h})].
    \end{equation*} 
    As $c(x_{t_0},x_{t_0})=0$ and $c(x_{t_0},x_{t_0+h})\ge 0$, we can  drop  the  (1-s)-term  of the r.h.s., and obtain
    \begin{equation*}
        c(\gamma_s,x_{t_0+h}) - c(\gamma_s,x_{t_0})\geq s \,[c(x,x_{t_0+h})-c(x,x_{t_0})].
    \end{equation*}
    Divide by $h$ and take the  $\limsup$ as $h\downarrow 0$  to obtain
    \begin{equation}
        \frac{\de^+ c(\gamma_s, x_t)}{\de t}_{|t=t_0}\geq s \frac{\de^+ c(x, x_t)}{\de t}_{|t=t_0}  . 
    \end{equation}
    Dividing by $s$ and taking the $\liminf$  for $s\downarrow 0$,  we obtain \eqref{eq:decreasing_slope}. 

    Consider now  \eqref{eq:convexity_g_c}, divide by $t$, and take the   $\liminf$ as  $t\downarrow 0$ to obtain
    \begin{equation}\label{eq:evi_diff_local}
       \frac{\de^+}{\de t} c(x, x_t)_{|t=t_0}\stackrel{\eqref{eq:decreasing_slope}}{\le}  [\dot x, \gamma]_{c,t_0} \stackrel{\eqref{eq:evi_local}}{\le}\phi'(x_{ t_0};\gamma)  \stackrel{\eqref{eq:convexity_g_c}}{\le} \phi(x)-\phi(x_{t_0})-\lambda c(x, x_{t_0})
    \end{equation}
    so $x(\cdot)$ satisfies \eqref{eq:evi_diff},  hence $x(\cdot) \in  \EVI_\lambda(X,c,\phi)$.
\end{proof}

\begin{corollary}   Let $0\le \lambda_1 \le \lambda_2$, $\phi:X\to (-\infty, + \infty]$, $c$ be a cost satisfying \eqref{eq:addass}, and $\sigma$ be compatible with the pair $(c,\phi)$.  Then, $EVI_{\lambda_2}(X,c,\phi)\subset EVI_{\lambda_1}(X,c,\phi)$. 
Conversely, for $(X,c)$  being a   NNCC space, if $x(\cdot) \in  EVI_{\lambda_1}(X \times X,c,\phi)$ and $\phi$ is $\lambda_2$-convex along variational $c$-segments as in \eqref{eq:convexity_g_c}, then $x(\cdot) \in  EVI_{\lambda_2}(X,c,\phi)$.
\end{corollary}
\begin{proof}
    The first part  follows immediately from  \eqref{eq:evi_diff}. The converse stems from \eqref{eq:evi_diff_local} applied to $\lambda=\lambda_2$, since \eqref{eq:evi_local}  is independent of $\lambda$ in the case $\lambda \geq 0$.  
\end{proof}

\section{Existence of EVI solutions as limit of splitting schemes}
\label{sec:f+g}

In this section, we are interested in describing a natural  minimizing  movements scheme that converges to the EVI. Unlike previous works, we will not focus  solely  on the implicit (Euler) iteration, but also on a more general splitting scheme. The key idea, following \cite{leger2023gradient}, is to perform a {\it majorization-minimization} of the functional $\phi:=f+g$ where $f:X\to\R$ and $g: X\to (-\infty, +\infty]$, noticing that for any set $X$, $\tau>0$, and any cost function $c:X\times  X\to \R$ we have
\begin{equation}\label{eq:maj-min}
    \phi(x)=f(x)+g(x)\le g(x)+\frac{c(x,y)}{\tau}+f^{c/\tau}(y)=: \Phi(x,y)
\end{equation}
using the $c$-transform $f^{c/\tau}(y):=\sup_{x'\in X}f(x')-\frac{c(x',y)}{\tau}$. Given $x_0\in X$, we then perform an alternating minimization of $\Phi$, assuming that the iterates exist,
\begin{align}
    y_{n+1}^\tau \in &\argmin_{y\in  X} g(x_n^\tau)+\frac{c(x_n^\tau,y)}{\tau}+f^{c/\tau}(y), \label{eq:iteration_y}\\
    x_{n+1}^\tau \in &\argmin_{x\in X}g(x)+\frac{c(x,y_{n+1}^\tau)}{\tau}+f^{c/\tau}(y^\tau_{n+1}). \label{eq:iteration_x}
\end{align}
 The alternating minimization of $\Phi$ does not provide a minimizer of $\phi$ in general. This is however the case if $\phi$ is $c/\tau$-concave (\Cref{def:c-transform}), which entails that $\phi(x)=\inf_{y\in X}\Phi(x,y)$.  

 In the relevant  subcase    $f=0$   we only have implicit iterations. If $c(\cdot,y)$ is lower bounded for each $y\in  X$, then, introducing $\tilde c(x,y):=c(x,y)-\inf_{x'\in X}c(x',y)$, we have that  
$\tilde c\ge 0$ and $\Phi(x,y)=g(x)+\frac{\tilde c(x,y)}{\tau}$ in \eqref{eq:maj-min},  which clearly defines a minimizing movement scheme. To ensure that no information about $g$ is lost, we require that  $0=\inf_{y\in  X  }\tilde c(x,y)$.

In \cite{leger2023gradient} Léger and the first-named author extensively studied the discrete iterations with $\tau>0$ fixed and proved convergence of $(\phi(x_n^\tau))_{n\in \N}$ to the infimum of $\phi$ under  the so-called {\it five-point}   property. We are instead interested  in  taking the limit $\tau\to 0$ and recovering  a trajectory in  $\EVI_\lambda(X,c,\phi)$.  In \Cref{sec:scheme}, we prove that this is possible if  
$f$ and $g$ satisfy that $f$ is $\nicefrac{c}{\tau}$-concave for all small $\tau$ and  some 
cross-convexity property from  
\cite{leger2023gradient}  holds. The latter reduces  to a form of discrete EVI. Similarly to the discrete EVI in the metric case \cite{Ambrosio2012userguide}, this cross-convexity is implied by a more readable compatibility condition between the functional and the cost,  described by  
some convexity of two functionals along the same curve, see \Cref{sec:compatibility}. 

Since we consider separate properties on $f$ and $g$ due to the splitting structure, we consider also two auxiliary variables $\xi,z\in  X $ corresponding to \eqref{eq:iteration_y} with $f=0$ and \eqref{eq:iteration_x} with $g=0$. Their role is made  clear  in \Cref{table:scheme_fg}.

\subsection{Definition of the scheme through alternating minimization}\label{sec:scheme}

We start with the definitions of the fundamental objects we are going to use. Throughout the section, $X$ will be  a  nonempty  set  and $c:X \times  X   \to \R$ will be a given cost.  Note that in the first part of this section we do not  need to  assume the cost to be non-negative, as we did in the previous Section \ref{sec:EVI}. 

\begin{definition}[$c$-transform, $c$-concavity] \label{def:c-transform}  
    The \emph{$c$-transform} of a function $U\colon X\to\R$ is the function $U^c\colon  X \to(-\infty, +\infty]$ defined by 
        \begin{equation} \label{eq:def-c-transform}
            U^c(y):=\sup_{x\in X} U(x)-c(x,y),  \quad y \in X.
        \end{equation}
    We say that a function $U\colon X\to\R$ is $c$-concave if there exists a function $h\colon  X \to (-\infty, +\infty]$ such that
    \begin{equation}\label{eq:def_c_concav}
         U(x)=\inf_{y\in  X }c(x,y)+h(y)\quad \text{ for every } x \in X.
    \end{equation}
     Note that, if $U$ is $c$-concave, then $U^c$ is the smallest $h$ such that \eqref{eq:def_c_concav} holds.
\end{definition}
Since $c$ has finite values, $c$-concave functions cannot take the value $+\infty$. The next lemma shows that a $c$-concave function is also $c/\tau$-concave for $\tau \in (0,1]$, provided the cost satisfies a suitable condition. 
\begin{lemma}
    Let $c':X\times  X \to \R$ be  given.  
    Every $c$-concave function is $c'$-concave if and only if $ x \mapsto c(x,y)$ is $c'$-concave for every $y \in  X $. In particular, given $\tau \in (0,1]$,  any  $c$-concave function is $c/\tau$-concave if and only if, for all $y\in  X $, there exists $h_y: X \to (-\infty, +\infty]$ such that $\tau c(\cdot,y)=\inf_{y'\in  X } c(\cdot,y')+h_y(y')$.

    If $(X,d)$ is a metric space,   such that there exist {\it $\epsilon$-midpoints} in the following sense 
    \begin{equation}
        \forall \epsilon>0,\, \forall \, x\in X,\, x'\in X, \, \exists y\in X , \, \max(d(x,y),d(x',y))\le \frac12 d(x,x')+\epsilon,
    \end{equation}
      then $\frac{d^2(\cdot,x)}{2^n}$ is $d^2$-concave for all $x\in X$ and $n\in\N$. 
\end{lemma}
\begin{proof} Suppose every $c$-concave function is $c'$-concave. For every $y \in  X $, the function $x \mapsto c(x,y)$ is $c$-concave (it is enough to take $h$ as the indicator function of $\{y\}$ in the definition of $c$-concavity) whence it is $c'$-concave.

Suppose now that $x \mapsto c(x,y)$ is $c'$-concave for every $y \in  X $. This is equivalent to say that, for every $y \in  X $, there exists a function $h_y:  X  \to (-\infty, + \infty]$ such that $c(x,y) = \inf_{y' \in  X } c'(x,y')+ h_y(y')$. If $U:X \to \R$ is $c$-concave, there exists $h: X  \to \R\cup\{+ \infty\}$ such that 
\[  U(x)=\inf_{y\in  X }c(x,y)+h(y)\quad \text{ for every } x \in X.\]
We deduce
\begin{align*}
    U(x)= \inf_{y \in  X  } \left [ \inf_{y' \in  X } c'(x,y')+h_y(y') \right ] + h(y) =
    \inf_{y' \in  X } \left \{ c'(x,y')+ \inf_{y \in  X } \left [ h_y(y')+h(y) \right] \right \}, \quad x\in X.
\end{align*}
Notice that the quantity $h'(y'):= \inf_{y \in  X } \left [ h_y(y')+h(y) \right]$ cannot attain the value $-\infty$ since otherwise we would obtain $U(x)=-\infty$, in contrast with $U(x) \in \R$.

Let $(X,d)$ be an intrinsic metric space. By the triangle  inequality, we have for any $x,x',y\in X$
    \begin{equation*}
        \frac12 d^2(x,x') \le d^2(x,y)+d^2(x',y).
    \end{equation*}
    Taking the infimum over $y$ we deduce that $\frac12 d^2(x,x') \le \inf_{y\in X}d^2(x,y)+d^2(x',y)$.  Take  $y_\epsilon$ to be an $\epsilon$-midpoint. We get $d^2(x,y_\epsilon)+d^2(x',y_\epsilon) \le 2(\frac12 d(x,x')+\epsilon)^2$. Hence taking the limit  as  $\epsilon\downarrow 0$, we obtain $\frac12 d^2(x,x')= \inf_{y\in X}d^2(x,y)+d^2(x',y)$.

 Replacing $d$ by $\frac{d}{\sqrt{2}}$, the same proof shows that $\frac14 d^2(x,\cdot)$ is $\frac{d^2}{2}$ concave, and an immediate induction gives the result for all $2^n$ with $n\in\N$. 
\end{proof}

While $c$-concavity is a classical notion, here used for explicit schemes, we now introduce the key notion of \cite{leger2023gradient}: cross-convexity. It is based on iterates and should be seen as requiring the analogue of the discrete EVI in gradient flows on metric spaces \cite[Corollary 4.1.3]{ags08}.

\begin{definition} Let $U:X \to \R$ be a function and let $x_0 \in X$ and $y_0 \in  X $. We define the (possibly empty) sets
\begin{align*}
\mathcal{P}_c(U,y_0)&:= \argmin_{x \in X} \{ c(x,y_0) - U(x) \} ,\\
\mathcal{Q}_c(U,x_0)&:= \argmin_{y \in  X } \{ c(x_0,y) +U^c(y) \}  ,\\
\mathcal{R}_c(x_0)&:= \argmin_{y \in  X } c(x_0,y)  ,\\
\mathcal{S}_c(x_0)&:=  \{\xi_0\in  X  \, | \, x_0\in\argmin_{x \in X} c(x,\xi_0) \} .
\end{align*}
\end{definition}

\begin{definition}[$c$-cross-concavity/convexity] Let $\mu\ge0$ be given. A function $U: X \to \R$ is said to be 
\begin{itemize}
\item \emph{$\mu$-strongly $c$-cross-concave} if 
\[ U(x)-U(x_1) \le \left [ c(x,y_0)-c(x_1, y_0)\right ] -  \left [ c(x,y_1)-c(x_1, y_1) \right ] -  \mu \left [ c(x,y_1)-c(x_1, y_1) \right ] \]
for every $(x,y_0) \in X\times  X $ and every $x_1 \in \mathcal{P}_c(U,y_0)$, $y_1 \in \mathcal{R}_c(x_1)$;
\item {\it $\mu$-strongly $c$-cross-convex} if 
 \[ U(x)-U(x_0) \ge -\left [ c(x,\xi_0)-c(x_0, \xi_0)\right ] +  \left [ c(x,y_0)-c(x_0, y_0) \right ] +  \mu \left [ c(x,\xi_0)-c(x_0, \xi_0) \right ] \]
 for every $(x,x_0) \in X \times X$ and every $\xi_0 \in \mathcal{S}_c(x_0)$, $y_0 \in \mathcal{Q}_c(U,x_0)$.
\end{itemize}
\end{definition}

Because of the definition of the iterates  and of  the focus on $y_0$ and $\mathcal{P}_c$ for cross-concavity and on $x_0$ and $\mathcal{Q}_c$ for cross-convexity, the two notions are  not symmetric: having $-U$ $c$-cross-concave does not imply that $U$ is $c$-cross-convex. In \Cref{sec:compatibility}, we provide some simpler curve-based sufficient assumptions for these two notions to hold.

\begin{continuance}{ex:distance} For $\tau>0$ and $c=\frac{d^2}{2\tau}$, $c$-cross-concavity is precisely the discrete EVI considered in \cite[Corollary 4.1.3]{ags08}. Cross-convexity is similar  concept,  but applies to an explicit scheme. The $\frac{d^2}{2\tau}$-concavity can be understood as a  local  $\frac{d^2}{2\tau}$-bound on the  growth.
\end{continuance}

\begin{continuance}{ex:KL} Let $X$ be a subset of a Banach space, $X^*$ the dual of the closure of its span, and assume there exists a strictly convex $u:X\to\R$ such that for all $y\in X$ there exists $u'(y)\in X^*$ such that for all $x$
\begin{equation}\label{eq:bregman_div}
    c(x,y)=u(x)-u(y)-\bracket{u'(y),x-y}.
\end{equation}
In particular if $u'(y)$ is the Gateaux derivative of $u$ at $y$, then \eqref{eq:bregman_div} states that $c$ is the Bregman divergence of $u$. The $\KL$ is not Gateaux differentiable for nondiscrete $\X\subset \R^d$, but \eqref{eq:bregman_div} still holds, see \cite[Example 2]{aubin2022mirror_measures} with $u(\mu)=\KL(\mu|\rho)$ and $u'(\mu)=1+\ln(\frac{ \de  \mu}{ \de  \rho})$. 

Cross-concavity is the three-point inequality in mirror descent, see e.g \cite[Lemma 3.2]{chen1993convergence}. It reads similarly to cross-convexity which takes the form: for all $x_0,x\in X$, since $\mathcal{S}_c(x_0)=\{x_0\}$ by strict convexity of $u$,
\begin{equation}\label{eq:bregman_div2}
    [U(x)-\mu u(x)]-[U(x_0)-\mu u(x_0)] \ge \bracket{u'(y_0)-u'(x)-\mu u'(x_0),x-x_0}
\end{equation}
with $y_0\in \mathcal{Q}_c(U,x_0)$. On the other hand, $c$-concavity reads $U(x)=\inf_{y\in  X }c(x,y)+h(y)$ which implies that $U-u$ is concave upper semicontinuous as an infimum of affine functions.

\end{continuance}

We work under the following fundamental hypotheses, which will be progressively reinforced by additional assumptions when needed.

\begin{assumption}\label{ass:ass2}  Let  $c:X \times  X \to \R$, $f:X \to \R$, $g: X \to (-\infty, + \infty]$  be given  and   assume that there exist $\bar{\tau}>0$ and $\lambda_f, \lambda_g \in \R$ such that:
\begin{enumerate}
\item $f$ is $\lambda_f \tau$-strongly $c/\tau$-cross-convex and $c/\tau$-concave for every $0<\tau< \bar{\tau}$.
\item $-g$ is $\lambda_g \tau$-strongly $c/\tau$-cross-concave for every $0<\tau< \bar{\tau}$.
\item for every $x_0 \in X$ there exists $y_0 \in \mathcal{Q}_{c/\tau}(f,x_0)$ and $x_1 \in \mathcal{P}_{c/\tau}(-g,y_0)$ such that $\mathcal{S}_c(x_1)\ne \emptyset$  and $ \mathcal{R}_c(x_1) \ne \emptyset$.
\end{enumerate}
\end{assumption}
Under the above Assumption \ref{ass:ass2}, starting from $x_0 \in X$ we can implement the following scheme: for $i \in  \N $ and $0<\tau< \bar{\tau}$ we define
\begin{equation}\label{eq:scheme2}
x^\tau_0:=x_0, \ \ y_i^\tau \in \mathcal{Q}_{c/\tau}(f,x_i^\tau), \ \ x_{i+1}^\tau \in \mathcal{P}_{c/\tau}(-g, y_i^\tau), \ \ \xi_i^\tau \in \mathcal{S}_c(x_i^\tau), \ \ z_i^\tau \in \mathcal{R}_c(x_i^\tau).
\end{equation}
\begin{figure}[h!]
\begin{tabular}{lccccc}\label{table:scheme_fg}
\hspace{5mm}& $y_{i-1}^\tau$ &  &  &  & $y_{i}^\tau$  \\ 
  & &  $\stackrel{\mathcal{P}_{c/\tau}(-g)}{\searrow}$  &  & $\stackrel{\mathcal{Q}_{c/\tau}(f)}{\nearrow}$ &  \\ 
& &  & $x_i^\tau$  &  &  \\ 
  &  & $\stackrel{\mathcal{R}_c}{\swarrow}$ &  & $\stackrel{\mathcal{S}_c}{\searrow}$  &     \\ 
\hspace{5mm} & $\xi_i^\tau$  &  &  &  & $z_i^\tau$
\end{tabular}
\caption{Diagram of the iterates of the splitting scheme \eqref{eq:scheme2}.
}
\end{figure}

    While $x_{i+1}^\tau \in \mathcal{P}_{c/\tau}(-g, y_i^\tau)$ clearly corresponds to an implicit iteration on $g$, we refer to \cite[Section 3]{leger2023gradient} for the interpretation of $y_i^\tau \in \mathcal{Q}_{c/\tau}(f,x_i^\tau)$ as an explicit iteration on $f$. 
    The intuition for this result is that $y_i^\tau \in \mathcal{Q}_{c/\tau}(f,x_i^\tau)$ is chosen so that the upper bound $f(x)\le\frac{c(x,y_i^\tau)}{\tau} +f^c(y_i^\tau)$ becomes an equality at $x_i^\tau$ under the $c/\tau$-concavity assumption. In differentiable settings, the envelope theorem then gives the relation $\tau \nabla f(x_i^\tau)= \nabla_x c(x_i^\tau,y_i^\tau)$. Hence $y_i^\tau$ is the next iterate of an explicit scheme starting at $x_i^\tau$.  In particular, it is shown in \cite{leger2023gradient} that,  if $X$ is an Euclidean space,  for $c/\tau$-concave $f$, this iteration corresponds to gradient, mirror or Riemmanian descent  for   $c$ taken respectively  to be   the squared Euclidean norm, a Bregman divergence, or a squared Riemannian distance.

\begin{definition}\label{def:interp2} Under Assumption \ref{ass:ass2}, for $0 < \tau < \bar{\tau}$,  given points $(x_i^\tau)_i \subset X$ and $(y_i^\tau)_i, (z_i^\tau)_i \subset  X $ as in \eqref{eq:scheme2}, we define the following interpolating curves 
$\bar{x}^\tau, 
\bar{y}^\tau, \bar{z}^\tau :[0,+\infty) \to  X $ as 
\[ \bar{x}^\tau_t:= x^\tau_{\floor{t/\tau}}, \quad \bar{y}^\tau_t:= y^\tau_{\floor{t/\tau}}, \quad \bar{z}^\tau_t:= z^\tau_{\floor{t/\tau}} \quad t \in [0,+\infty). \]
\end{definition}

 Our aim is to prove the following theorem.

\begin{theorem}[Existence of EVI solutions  via  splitting scheme]\label{thm:existence_limit_scheme_fg}   Let $c$ be a cost satisfying \eqref{eq:addass} with  Assumption \ref{ass:ass2} in place;   suppose in addition that 
\begin{enumerate}
    \item[\rm (1)] $\phi=f+g$ is lower bounded and that $\lambda_f, \lambda_g \ge 0$;
    \item[\rm (2)]\label{it:c_continuous} 
    there exists a  Hausdorff  topology $\sigma$ on $X$ such that $\sigma$ is compatible with $( c ,\phi)$ and  $c$  is $\sigma$-continuous;
    \item[\rm (3)] Either
    \begin{enumerate}
        \item $\sigma$ is Cauchy-compatible with  $( c,\phi)$ and $c$ is symmetric, or
        \item  the sublevel sets  of  $\phi$ are  $\sigma$-sequentially  compact.
    \end{enumerate}
\end{enumerate}
Then, for every $x_0 \in X$ and $\tau \in (0, \bar{\tau})$ the family of curves $(\bar{x}^{\tau/2^n})$ (constructed according to  Definition \ref{def:interp2}  starting from $x_0$) converges (up to a subsequence in case (b)) pointwise (w.r.t.~$\sigma$) as $\tau \downarrow 0$ to a $\sigma$-continuous curve $x:[0,+\infty) \to X$ which satisfies
\begin{multline}
c(x, x_t )-c(x,  x_s  ) + (\lambda_f+\lambda_g)\int_{s}^t c(x, x_r  ) \de r \\
 \le (t-s)\phi(x)-\int_s^t \phi(x_r) \de r \quad \forall \, 0 \le s \le t, \  x \in X.    \label{eq:evi_fg}
\end{multline}
If $\EVI_\lambda(X,c,\phi)$  contains  a unique trajectory $x(\cdot)$ starting from $x_0$, then, $\bar{x}^{\tau}(\cdot)$ $\sigma$-converges pointwise to this $x(\cdot)$, which does not depend on $\tau$. In case (a), we  also have the following error estimate
\begin{equation}\label{eq:error_estimate_f+g}
    c \left (\bar{x}^{\tau}_t, x_t \right ) \le 2\tau (\phi(x_0)-\inf \phi ) \quad \forall t \ge 0, \, \tau \in (0, \bar{\tau}).
\end{equation}
\end{theorem}
  The proof of Theorem \ref{thm:existence_limit_scheme_fg} is deferred to Section \ref{sec:f+gproof}. In terms of assumptions, the key differences with \Cref{sec:EVI} are that we assume $\phi$ to be lower bounded in order to have a lower bound on $(\phi(x^\tau_n))_n$ independent of $\tau$ and $n$. We also require $c$ to be $\sigma$-continuous so as to take limits in $-c(x,\bar z^\tau_t)$, where lower semicontinuity of $c$ would not  be sufficient  to conclude.

 Before moving on, let us mention that the assumptions of Theorem \ref{thm:existence_limit_scheme_fg} may be weakened in some specific situations. A specific setting where this seems to be possible is the case in which $c$ dominates a (power of a) distance $d^2$. In this case, the a-priori bounds obtained below (see \eqref{eq:intevi}, for instance) would guarantee that the approximating trajectories are equicontinuous w.r.t to $d^2$. In combination with a compactness assumption on the sublevels of $\phi$, this would allow to use an Ascoli--Arzel\`a argument to obtain convergence. 
In the linear-space setting of $X$ being a Hilbert or a Banach space, one may even avoid asking for the compactness of the sublevel sets of $\phi$ and alternatively handle the limit passage by lower-semicontinuity arguments. We give below an application of Theorem \ref{thm:existence_limit_scheme_fg} to \Cref{ex:KL}. 
\begin{corollary}\label{coro:KL_flow}
    Let $\X\subset \R^d$ be a closed subset, $0<a<b<+\infty$, $\rho\in \calP(\X)$, and let $X=\{\mu\in \calP(\X) \, | \, \text{$\mu =h\rho$ with $h(\cdot)\in [a,b]$ $\rho$-a.s.} \}$. Let $\sigma_\rho$ be the  strong  $L^1(\rho)$ topology on $X$, and set $c=\KL$. Let $f,g: X \to \R$ be convex, $\sigma_\rho$-lower semicontinuous, and such that $g$ has $\sigma_\rho$ compact sublevel sets  and  is lower-bounded and $f$ is lower bounded. Assume furthermore that there exists $\bar{\tau}>0$ such that for every $0<\tau< \bar{\tau}$, there exists $h_\tau:X\to (-\infty,+\infty]$  satisfying for all $\mu \in X$ 
    \begin{equation}
        \frac{\KL(\mu | \rho)}{\tau} -f(\mu)=\max_{\mu'\in X} \left\langle \mu,\frac{1}{\tau}\ln\left(\frac{ \de \mu'}{ \de \rho}\right)\right\rangle -h_\tau(\mu').\label{eq:KL_c-concavity}
    \end{equation}
    Then, for every $x_0 \in X$ and $\tau \in (0, \bar{\tau})$ the family of curves $(\bar{x}^{\tau/2^n})$ (constructed according to  Definition \ref{def:interp2}  starting from $x_0$) converges up to a subsequence pointwise (w.r.t.~$\sigma$) as $\tau \downarrow 0$ to a $\sigma$-continuous curve $x:[0,+\infty) \to X$  satisfying  \eqref{eq:evi_fg}.
\end{corollary}
\Cref{coro:KL_flow}  proves  that the limit curve belongs $\EVI_\lambda(X,c, \phi)$. Based on the discussion around \eqref{eq:evi_local2}, taking derivatives $\nabla_{2,1}$ in \eqref{eq:bregman_div}, this flow  formally  corresponds to the mirror flow $$-[\nabla^2 (\KL(\cdot | \rho))(x_{t_0})]\dot x_{t_0} \in \partial \phi(x_{t_0})$$ restricted to the set $X$.

\begin{proof}
We just have to check the assumptions of Theorem \ref{thm:existence_limit_scheme_fg}. Since $c$ clearly  satisfies  \eqref{eq:addass} and (1), (2) and (3){\it b}  hold,  it only remains to check Assumption \ref{ass:ass2}.
Since  the sublevels  of $g$ are compact for  $\sigma_\rho$  and $g(\cdot)+\frac{\KL(\cdot| \mu_0)}{\tau}$ is $\sigma_\rho$-lower semicontinuous for all $\mu_0 \in X$, the set $\calP_{c/\tau}(-g,y_0)$ is  nonempty.  We use \Cref{lem:ced_cross-concave,lem:ced_cross-convex} below, along with the discussion  in Example \ref{ex:KL_sec3} below  on $c=\KL$. Equation \eqref{eq:KL_c-concavity} is just a rewriting of the $c/\tau$-concavity of \eqref{eq:def_c_concav}.
\end{proof}

\subsection{Sufficient conditions: compatibility of energy and cost}\label{sec:compatibility}

We now provide some more readable sufficient assumptions, inspired by \cite[Section 3.2.4]{Ambrosio2012userguide}, also known as C2G2 (Compatible Convexity along Generalized Geodesics) in \cite{Santambrogio2017}.

\begin{assumption}[Compatibility to obtain cross-concavity]\label{ass:ced_cross-concave}
The function $g:X \to (-\infty, + \infty]$ is such that there exists $\lambda\in\R$ such that, for any $y_0\in Y$ and $x_1\in X$ and $y_1\in \calR_c(x_1)$, and for all $x\in X$, there exists functions $\gamma:[0,1]\to X$ and $M:[0,1]\to (-\infty, +\infty]$ such that, for every $t\in[0,1]$
    \begin{align}
        g(\gamma(t)) &\le (1-t)g(x_1)+tg(x)-\lambda t [c(x,y_1)- c(x_1,y_1)]+M_t, \label{eq:ced_g} \\
        c(\gamma(t),y_0) &\le (1-t)c(x_1,y_0)+tc(x,y_0)- t [c(x,y_1)-c(x_1,y_1)]+M_t \label{eq:ced_g_c} 
    \end{align}
    with $\liminf_{t\downarrow 0}\frac{M_t}{t}=0$.  Specifically,  $\gamma(\cdot)$  fulfills 
    \begin{align}
        \liminf_{t\downarrow 0}\frac{g(\gamma(t))-g(x_1)}{t} &\le g(x)-g(x_1)-\lambda [c(x,y_1)- c(x_1,y_1)], \label{eq:ced_g2} \\
        \liminf_{t\downarrow 0}\frac{c(\gamma(t),y_0)-c(x_1,y_0)}{t} &\le c(x,y_0)-c(x_1,y_0)- [c(x,y_1)-c(x_1,y_1)]. \label{eq:ced_g_c2} 
    \end{align}
\end{assumption}

\begin{assumption}[Compatibility to obtain cross-convexity]\label{ass:ced_cross-convex}
   There exists $\bar \tau>0$ such that $f$ is $c/\tau$-concave for all $\tau\in (0,\bar{\tau})$. There exists $\lambda\in\R$ such that, for any $x_0\in X$, and for all $x\in X$ and $\tau\in (0,\bar{\tau})$, $\xi_0\in \mathcal{S}_c(x_0)$, $y_0\in \calQ_{c/\tau}(f,x_0)$, there exists functions $\gamma:[0,1]\to X$, $z:[0,1]\to \YY$ and $M:[0,1]\to (-\infty, +\infty]$ such that $\gamma(0)=x_0$, $z(0)=y_0$ and, for every $t\in[0,1]$, $f(\gamma(t))-f^{c/\tau}(z(t))=\frac{1}{\tau}c(\gamma(t),z(t))$ and
    \begin{align}
        f(\gamma(t)) &\le (1-t)f(x_0)+tf(x)-\lambda t[c(x,\xi_0)-c(x_0,\xi_0)]+M_t, \label{eq:ced_f} \\
        c(\gamma(t),z(t)) &\ge (1-t)c(x_0,z(t))+tc(x,z(t))-t[c(x,\xi_0)-c(x_0,\xi_0)]-M_t, \label{eq:ced_f_c} 
    \end{align}
     with $\liminf_{t\to 0} [c(x,z(t))-c(x_0,z(t))]\le c(x,z(0))-c(x_0,z(0))$ and $\liminf_{t\downarrow 0}\frac{M_t}{t}=0$.
\end{assumption}
\begin{remark}\label{rmk:compatibility} In  the   analysis, it is paramount to choose which functions should be convex along which curves.  Nonetheless, we mostly use assumptions for $t \to 0$ only.  
By introducing the function $t \mapsto M_t$ we intend to cover both convexity of $t\mapsto \phi(\gamma_t) - \lambda t^2 c(x,x_0)$ and of $t\mapsto \phi(\gamma_t) - \lambda c(\gamma_t,x_0)$ as already discussed around  \eqref{eq:convexity_g_c} after \Cref{prop:evi_local_2}. For instance, for \eqref{eq:ced_g}--\eqref{eq:ced_g_c} and $p>1$, taking $M_t= \max(1,\lambda) t^p[c(x,y_1)- c(x_1,y_1)]$, we recover the $(p,\lambda)$-convexity of Ohta and Zhao in the case $c=d^p$ \cite[Definition 4.1]{Ohta2023}. Actually summing \eqref{eq:ced_g} and \eqref{eq:ced_g_c} multiplied by $\nicefrac{1}{\tau}$, one gets a form of local convexity of $g(\cdot)+\frac{c(\cdot,y_0)}{\tau}$ in the spirit of \cite[Assumption 4.0.1]{ags08} or its analogue \cite[Assumption 4.3]{Ohta2023} which both give the discrete EVI with the same proof as below.

\end{remark}

\begin{lemma}\label{lem:ced_cross-concave}
     Let $\tau >0$ and assume that for all $y_0\in \YY$ the set $\calP_{c/\tau}(-g,y_0)$ is nonempty. Then, \Cref{ass:ced_cross-concave} implies the $\lambda\tau$-strong $c/\tau$-cross-concavity of $-g$.
\end{lemma}
\begin{proof}
    Fix $ x_1 \in \calP_{c/\tau}(-g,y_0)$, multiply \eqref{eq:ced_g_c} by $1/\tau$ and sum with \eqref{eq:ced_g} to obtain
    \begin{align}
        g(x_1) +\frac{1}{\tau}c(x_1,y_0) \nonumber
        &\le g(\gamma(t)) +\frac{1}{\tau}c(\gamma(t),y_0)\\
        &\le (1-t)g(x_1)+tg(x)-\lambda t c(x,y_1) \nonumber \\
        &\hspace{5mm}+\frac{1-t}{\tau}c(x_1,y_0)+\frac{t}{\tau}c(x,y_0)- \frac{t}{\tau} [c(x,y_1)- c(x_1,y_1)]+(1+\frac{1}{\tau})M_t
    \end{align}
    Observe that the terms not depending on $t$ simplify, divide by $t$, and take the limit $t\to 0$. We obtain that, for every $x\in X$ and every $x_1 \in \mathcal{P}_c(-g,y_0)$  and   $y_1 \in \mathcal{R}_c(x_1)$,
    \begin{equation}
        g(x_1) +\frac{1}{\tau}c(x_1,y_0)\le g(x)+\frac{1}{\tau}c(x,y_0)-\frac{\lambda\tau+1}{\tau}[c(x,y_1)- c(x_1,y_1)]
    \end{equation}
    which is precisely the cross-concavity of $-g$.
\end{proof}

\begin{lemma}\label{lem:ced_cross-convex}
    \Cref{ass:ced_cross-convex} implies the $c/\tau$-cross-convexity of $f$ for $\tau \in (0,\bar{\tau})$.
\end{lemma}
\begin{proof}
    Fix $\tau \in (0,\bar{\tau})$. Use that $f(x_0)-\frac{c(x_0,z(t))}{\tau}\le f^{c/\tau}(z(t))=f(\gamma(t))-\frac{c(\gamma(t),z(t))}{\tau}$, multiply \eqref{eq:ced_f_c} by $1/\tau$ and sum with \eqref{eq:ced_f} to obtain
    \begin{multline}
    f(x_0)\le 
        f(\gamma(t))-\frac{c(\gamma(t),z(t))}{\tau}+\frac{c(x_0,z(t))}{\tau}\\
        \le (1-t)f(x_0)+tf(x)+\frac{1-\lambda\tau}{\tau} t[c(x,\xi_0)-c(x_0,\xi_0)]\\
        +\frac{t}{\tau}c(x_0,z(t))-\frac{t}{\tau}c(x,z(t))+(1+\frac{1}{\tau})M_t.
    \end{multline}
    The term $f(x_0)$ cancels out. We then divide by $t$ and take the limit $t\to 0$. We obtain that, for every $x \in X$, every $\xi_0\in \mathcal{S}_c(x_0)$,  and  $y_0=z(0) \in \calQ_{c/\tau}(f,x_0)$,
    \begin{equation}
        f(x_0)\le f(x)+\frac{1-\lambda\tau}{\tau}[c(x,\xi_0)-c(x_0,\xi_0)]+\frac{1}{\tau}\liminf_{t\to 0}[c(x,z(t))-c(x_0,z(t))].
    \end{equation}
  As  the last term is bounded by $( c(x,z(0))-c(x_0,z(0)))$,  this proves  the cross-convexity of $f$.
\end{proof}

\begin{remark}[Explicit vs. implicit] Notice that the existence of $z(t)$ in \Cref{ass:ced_cross-convex} implies that $f$ is $c/\tau$-concave over $\gamma([0,1])$, but since $x_0=\gamma(0)$ is left free, we should nevertheless request $c/\tau$-concavity over the whole $X$. Moreover, the $c/\tau$-concavity of $f$ does not imply the existence of $z(t)$ which we hence have to assume. Informally, this second trajectory $z(t)$, the smoothness required by the $\liminf$, and the role of $\tau$ do not appear in \Cref{ass:ced_cross-concave}  as  implicit schemes for $g$  do not rely on smoothness  and require only a lower bound on the Hessian of $g$.  On the contrary,  
explicit schemes  require  smoothness and the Hessian of $f$ is assumed to fulfill a lower and an upper bound.    We refer to \cite{leger2023gradient} for more  details  on this aspect.
\end{remark}

\begin{continuance}{ex:distance} For geodesic metric spaces $(X,d)$ and $c=d^2$, \eqref{eq:ced_g_c} and \eqref{eq:ced_f_c} are implied by Alexandrov curvature conditions, for $\gamma(\cdot)$  being  a geodesic between $x_0$ and $x$.

 In particular, if $X$ is a Hilbert space and $c(x,y)=\frac{\|x-y\|^2}{2}$, this  implies that  \eqref{eq:ced_f_c}, with $c(x_0,z(0))=0$, and \eqref{eq:ced_g_c}, with $c(x_1,z_1)=0$,  are both  satisfied and  turn   into equalities. We then see that \eqref{eq:ced_f} and \eqref{eq:ced_g} are nothing  but  the usual $\lambda$-convexity of $f$ or $g$ for  $M_t=t^2[c(x,y_1)- c(x_1,y_1)]$. If $X$ is  a smooth reflexive Banach space, then \Cref{ass:ced_cross-concave}  for $g={\|\cdot\|^2}/{2}$, setting $f=0$ (resp.\ \Cref{ass:ced_cross-convex} for $f={\|\cdot\|^2},{2}$, setting $g=0$), implies that $X$ is Hilbert. Indeed, using \Cref{thm:existence_limit_scheme_fg}  one can find a trajectory in  $\EVI_\lambda(X,c, g)$ starting from any $x_0\in X$,  and we  can apply \cite[Corollary 4.5]{vonRenesse2012}.

  When considering a geodesic space $X$ with nonzero curvature, an interesting phenomenon arises. If $X$ is nonpositively curved (NPC) in the sense of Alexandrov, we can consider a geodesic $\gamma(t)$ for \Cref{ass:ced_cross-concave}, but not for \Cref{ass:ced_cross-convex}; the opposite holds if $X$ is positively curved (PC). In other words, when $X$ is NPC, every $\lambda$-convex function $g$ satisfies \Cref{ass:ced_cross-concave}, and one finds an EVI solution by \Cref{thm:existence_limit_scheme_fg}, as also stated in \cite[Theorem 3.14]{MuratoriSavare}. Conversely, in PC spaces, \Cref{ass:ced_cross-concave} implies that $-g$ is $c/\tau$-cross-concave for $\tau \in (0, 1/\lambda^-)$, which in turn implies that $g$ is $\lambda$-convex, as shown in \cite[Proposition 4.14]{leger2023gradient}.

If $X$ is the Wasserstein space equipped with the distance $W_2$, which is positively curved, \cite[Remark 9.2.8]{ags08} discusses how \Cref{ass:ced_cross-concave} with $M_t = t^2[c(x,y_1)- c(x_1,y_1)]$ corresponds to convexity along generalized geodesics. This assumption implies \cite[Assumption 4.0.1]{ags08}, which leads to the discrete EVI and consequently to the continuous EVI. However, they also emphasize that convexity along generalized geodesics is a stronger requirement, as it implies convexity along standard geodesics.
Note that \eqref{eq:ced_g_c} holds in both NPC spaces (along geodesics) and in NNCC spaces (along variational $d^2$-segments), with NNCC being a subclass of PC spaces.

In summary, explicit schemes are particularly well-suited for positively curved spaces and $c$-concave, $\lambda$-convex functions (which are therefore $c$-cross-convex) while in NPC spaces, $\lambda$-convex functions are more naturally handled via implicit schemes. The strong focus on implicit schemes in the Wasserstein space is likely motivated by the fact that some relevant functionals, such as the entropy, show some convexity but are not $W_2^2$-concave. Furthermore, as our discussion suggests, implicit schemes are conceptually simpler: they require only a single bound and demand less regularity. Finally, as shown in \cite[Theorem 3.11]{leger2024nonnegative}, the Wasserstein space is NNCC (as defined in \Cref{def:nncc}), so its structure is shaped not only by its positive curvature but also by its NNCC nature.
\end{continuance}

\begin{continuance}{ex:KL}\label{ex:KL_sec3} For $c$ as in \eqref{eq:bregman_div} with $X$ convex, in particular $c=\KL$, taking $\gamma(t)=(1-t)x_1+tx$, one can argue similarly to \cite[Example 2.9]{leger2024nonnegative} to show that $(X\times X,c)$ satisfies \Cref{def:nncc} with \eqref{eq:NNCC-ineq}  being  an equality. So \eqref{eq:ced_g_c} and \eqref{eq:ced_f_c} hold with equality. Taking $f$ and $g$ convex will then give \eqref{eq:ced_g} and \eqref{eq:ced_f}. The $c/\tau$-concavity on $f$ implies that $\frac{u}{\tau}-f$ is convex l.s.c., the existence of $z(t)$ corresponds to having a nonempty subdifferential at all points.

\end{continuance}

\begin{continuance}{ex:Sinkhorn} For $\Seps$, preliminary work following \cite{lavenant2024riemannian} seems to indicate that potential functionals of the form $\mathcal{V}(\mu)=\int V(x)d\mu(x)$ with $V$ convex $C^1$  are  candidates for satisfying \Cref{ass:ced_cross-concave}~or~\ref{ass:ced_cross-convex}.  A rigorous verification of this fact remains an open problem at this stage. 

\end{continuance}

\subsection{Existence of the flow for implicit schemes}\label{sec:f=0}

 In this section, we focus  on implicit iterations, i.e.\ $f=0$, and may hence  simplify \Cref{ass:ass2}  as follows. 

\begin{assumption}\label{ass:ass1}  Let  $c:X \times \YY\to \R$ and $g: X \to \R \cup {+\infty}$  be given  and  assume   that there exist $\bar{\tau}>0$ and $\lambda \ge 0$ such that
\begin{enumerate}
\item $-g$ is $\lambda \tau$-strongly $c/\tau$-cross-concave for every $0<\tau< \bar{\tau}$;
\item for every $x_0 \in X$ and every $0< \tau<\bar{\tau}$ there exists some $y_0 \in \mathcal{R}_c(x_0)$ such that $\mathcal{P}_{c/\tau}(-g,y_0) \ne \emptyset$.
\end{enumerate}
\end{assumption}

Under the above Assumption \ref{ass:ass1},  given an arbitrary starting point $x_0 \in X$ and $\tau \in (0,\bar{\tau})$ we consider the following scheme, which corresponds to \eqref{eq:scheme2} with $z_i^\tau=y_i^\tau$ since $f=0$.

\begin{definition}[Scheme for $\phi=g$]\label{def:scheme}
For all  $i \in  \N $ we iteratively define
\begin{equation}\label{eq:scheme}
x^\tau_0:=x_0, \quad y_i^\tau \in \mathcal{R}_c(x^\tau_i), \quad x_{i+1}^\tau \in \mathcal{P}_{c/\tau}(-g, y_i^\tau).
\end{equation} 
\end{definition}

Note that Assumption {\ref{ass:ass1}} guarantees the existence of two sequences $(x_i^\tau)_i \subset X$ and $(y_i^\tau)_i \subset  X $ solving the scheme \eqref{eq:scheme}. We  define the  interpolating curves $\bar{x}^\tau: [0,+\infty) \to X$ and $\bar{y}^\tau:[0,+\infty) \to  X $ as 
\begin{equation} \bar{x}^\tau_t:= x^\tau_{\floor{t/\tau}}, \quad \bar{y}^\tau_t:= y^\tau_{\floor{t/\tau}}, \quad t \in [0,+\infty). \label{def:interp} 
\end{equation}

\begin{lemma}\label{le:intevi} Under Assumption {\ref{ass:ass1}}, let $\bar{x}^\tau$ and $\bar{y}^\tau$ be  defined as in \eqref{def:interp}. For  any $0 \le s= m\tau < n\tau = t$,  we have that
\begin{multline}
    \frac{c(x,\bar{y}_t^\tau)-c(x,\bar{y}_s^\tau)}{t-s} + \frac{1}{t-s} \sum_{i=m}^{n-1}\left [c(x_{i+1}^\tau, y_i^\tau)-c(x_{i+1}^\tau, y_{i+1}^\tau) \right ] + \frac{\lambda \tau}{t-s} \sum_{i=m}^{n-1}c(x,y_{i+1}^\tau)\\
    \quad \le g(x) - \frac{\tau}{t-s} \sum_{i=m}^{n-1} g(x_{i+1}^\tau)\quad \forall x \in X.\label{eq:intevi}
    \end{multline} 
\end{lemma}

 Relation  \eqref{eq:intevi} is a discrete version of  the EVI in integral form  \eqref{eq:evi_int}, our goal is to show that we can indeed take the limit in $\tau$.

\begin{proof} Since $-g$ is $\lambda \tau$-strongly $c/\tau$-cross-concave, by the definition of the scheme, we have
\begin{align}
\frac{c(x,y_{i+1}^\tau)- c(x,y_i^\tau)}{\tau} + \frac{c(x_{i+1}^\tau, y_i^\tau)-c(x_{i+1}^\tau, y_{i+1}^\tau)}{\tau} + \lambda c(x,y_{i+1}^\tau) \nonumber\\
\quad \le g(x)-g(x_{i+1}^\tau) \quad \forall \, x \in X, \, i \in  \N .\label{eq:tobesum}
\end{align}
We sum \eqref{eq:tobesum} from $i=m$ to $i=n-1$ to get
\begin{multline}
    \frac{c(x,\bar{y}_t^\tau)-c(x,\bar{y}_s^\tau)}{\tau} + \frac{1}{\tau} \sum_{i=m}^{n-1}c(x_{i+1}^\tau, y_i^\tau) - \frac{1}{\tau} \sum_{i=m}^{n-1} c(x^\tau_{i+1},y^\tau_{i+1}) +\lambda \sum_{i=m}^{n-1}c(x,y_{i+1}^\tau) \\
    \le (n-m)g(x) - \sum_{i=m}^{n-1} g(x_{i+1}^\tau) \label{eq:poi}
\end{multline}
and dividing by $(n-m)$ we obtain \eqref{eq:intevi}.
\end{proof}

 If $c$ satisfies \eqref{eq:addass}, then  \eqref{eq:scheme} defines a descent scheme over $g$, as well as  a first step toward a Cauchy estimate involving the symmetrization of $c$.
\begin{lemma}[$g$ non-increasing and first Cauchy-like estimate]\label{lem:non-increasing}  Let $c$ be a cost satisfying \eqref{eq:addass}  and let  Assumption \ref{ass:ass1} hold. Then,   the map $t \mapsto g(\bar{x}^\tau_t)$ is non-increasing. Moreover, for any $0 \le n\tau=t$ and $p \in \N$, we have
    \begin{align}
&[c(\bar{x}_{t}^{\tau/2^p},\bar{y}_{t}^\tau)+c(\bar{x}_{t}^\tau,\bar{y}_{t}^{\tau/2^p})]\nonumber\\
 &  \qquad +\sum_{i=0}^{n-1}[c(\bar{x}_{i\tau}^{\tau/2^p},\bar{y}_{(i+1)\tau}^\tau)- c(\bar{x}_{(i+1)\tau}^{\tau/2^p},\bar{y}_{i\tau}^\tau) + c(\bar{x}_{i\tau}^\tau,\bar{y}_{(i+1)\tau}^{\tau/2^p})-c(\bar{x}_{(i+1)\tau}^\tau,\bar{y}_{i\tau}^{\tau/2^p})]\nonumber\\ 
 & \qquad  +2\sum_{i=0}^{n-1}c(\bar{x}_{(i+1)\tau}^\tau, \bar{y}_{i\tau}^\tau)+ 2\sum_{j=0}^{n2^p-1}c(\bar{x}_{(j+1)\tau/2^p}^{\tau/2^p}, \bar{y}_{j\tau/2^p}^{\tau/2^p})\nonumber\\
    & \quad \le \tau(2g(x_0)-g(\bar{x}_{t}^\tau)-g(\bar{x}_{t}^{\tau/2^p})).\label{eq:Cauchy_ineq_full}
\end{align}
\end{lemma}

\begin{proof}
Assumption \eqref{eq:addass} and the fact that $ y^\tau_i \in \mathcal{R}_c(x_i^\tau)$ imply that $c(x_i^\tau,y_i^\tau)=0$.
 As $x^\tau_{i+1} \in \mathcal{P}_{c/\tau}(-g,y_i^\tau)$ and $c\geq 0$ we have that 
 $$g (x_{i+1}^\tau) \stackrel{c\geq0}{\leq} \frac{c(x_{i+1}^\tau,y^\tau_i) }{\tau}+ g (x_{i+1}^\tau) \stackrel{x^\tau_{i+1} \in \mathcal{P}_{c/\tau}(-g,y_i^\tau)}{\leq} \frac{c(x_{i}^\tau,y^\tau_i)}{\tau} + g (x_{i}^\tau) \stackrel{c(x_i^\tau,y_i^\tau)=0}{=} g(x_i^\tau). $$
 This proves that $(g(x_i^\tau))_i$ is non-increasing. 
 
 Since $c \ge 0$  from \eqref{eq:addass}, by using \eqref{eq:poi} with $t=(i+1)\tau$, $s=i\tau$, $n=i+1$, and $m=i$, and recalling that $c(x^\tau_{i+1},y^\tau_{i+1})=0$ as $y^\tau_{i+1}\in {\mathcal R}_c(x^\tau_{i+1})$,  for all $i\in \N$  we get 
\begin{align}\label{eq:tobesum_simple2}
&\frac{c(x,\bar{y}_{(i+1)\tau}^\tau)- c(x,\bar{y}_{i\tau}^\tau)}{\tau} + \frac{c(x_{i+1}^\tau, y_i^\tau)}{\tau}  \le g(x)-g(\bar{x}_{(i+1)\tau}^\tau) \quad \forall \, x \in X.
\end{align}
 Write now \eqref{eq:intevi} for the time   step $\tau/2^p$,  and for the choices  $n=(i+1)2^p$ and $m=i2^p$.  Still   using that $c \ge 0$  and the fact that $c(x^\tau_{i+1},y^\tau_{i+1})=0$  we get  
\begin{align}
    &\frac{c( x ,\bar{y}_{(i+1)\tau}^{\tau/2^p})-c( x ,\bar{y}_{i\tau}^{\tau/2^p})}{\tau}+\frac{1}{\tau} \sum_{j=i2^p}^{(i+1)2^p-1}c(\bar{x}_{(j+1)\tau/2^p}^{\tau/2^p}, \bar{y}_{j\tau/2^p}^{\tau/2^p}) \nonumber\\
    &\quad \le g( x ) -g(\bar{x}_{(i+1)\tau}^{\tau/2^p}) \quad \forall \, x   \in X, \, i \in \N.\label{eq:poi_simple2}
\end{align}
 By summing   \eqref{eq:tobesum_simple2} with $x=\bar{x}_{i\tau}^{\tau/2^p}$ and \eqref{eq:poi_simple2} with $x =\bar{x}_{(i+1)\tau}^\tau$,  and then  \eqref{eq:tobesum_simple2} with $x=\bar{x}_{(i+1)\tau}^{\tau/2^p}$ and \eqref{eq:poi_simple2} with $ x= \bar{x}_{i\tau}^\tau$  we   obtain
\begin{align*}
    &c(\bar{x}_{i\tau}^{\tau/2^p},\bar{y}_{(i+1)\tau}^\tau)- c(\bar{x}_{i\tau}^{\tau/2^p},\bar{y}_{i\tau}^\tau) + c(\bar{x}_{(i+1)\tau}^\tau,\bar{y}_{(i+1)\tau}^{\tau/2^p})-c(\bar{x}_{(i+1)\tau}^\tau,\bar{y}_{i\tau}^{\tau/2^p})\\
    &\quad +c(\bar{x}_{(i+1)\tau}^\tau, \bar{y}_i^\tau)+\sum_{j=i2^p}^{(i+1)2^p-1}c(\bar{x}_{(j+1)\tau/2^p}^{\tau/2^p}, \bar{y}_{j\tau/2^p}^{\tau/2^p})\le \tau(g(\bar{x}_{i\tau}^{\tau/2^p})-g(\bar{x}_{(i+1)\tau}^{\tau/2^p})), \\[1mm]
    &c(\bar{x}_{(i+1)\tau}^{\tau/2^p},\bar{y}_{(i+1)\tau}^\tau)- c(\bar{x}_{(i+1)\tau}^{\tau/2^p},\bar{y}_{i\tau}^\tau) + c(\bar{x}_{i\tau}^\tau,\bar{y}_{(i+1)\tau}^{\tau/2^p})-c(\bar{x}_{i\tau}^\tau,\bar{y}_{i\tau}^{\tau/2^p})\\
    &\quad +c(\bar{x}_{(i+1)\tau}^\tau, \bar{y}_i^\tau)+\sum_{j=i2^p}^{(i+1)2^p-1}c(\bar{x}_{(j+1)\tau/2^p}^{\tau/2^p}, \bar{y}_{j\tau/2^p}^{\tau/2^p}) \le \tau(g(\bar{x}_{i\tau}^\tau)-g(\bar{x}_{(i+1)\tau}^\tau)).
\end{align*}
 Taking the sum of   the above inequalities and  summing  from $i=0$ to $i=n-1$  we obtain \eqref{eq:Cauchy_ineq_full}.
\end{proof}

\begin{proposition}[Cauchy estimate]\label{prop:cauchy}  Let $c$ be a cost satisfying \eqref{eq:addass}  and let  Assumption \ref{ass:ass1} hold. Assume additionally  that $c$     decomposes as $c=c_1+c_2$ with 
\begin{enumerate}
    \item  $c_1\geq 0$  symmetric: $c_1(x_1,x_2)=c_1(x_2,x_1)$ for all $x_1,\, x_2 \in X$;  
    \item $c_2$ fulfilling the triangle inequality: $$c_2(x_1,x_2)\leq c_2(x_1,x_3)+ c_2(x_3,x_2)\quad \forall x_1,\, x_2, \, x_3 \in X.$$  
\end{enumerate}
Then, for every $0 \le n \tau = t$ and every $p \in \N$, we have
\[ c(\bar{x}_{t}^{\tau/2^p},\bar{x}_{t}^\tau)+c(\bar{x}_{t}^\tau,\bar{x}_{t}^{\tau/2^p}) \le \tau(2g(x_0)-g(\bar{x}_{t}^\tau)-g(\bar{x}_{t}^{\tau/2^p})).\]
    
\end{proposition}

\begin{proof}  From the nondegeneracy of $c$  coming from \eqref{eq:addass}  we get that $x_i^\tau=y_i^\tau$ for all $\tau \in (0,\bar{\tau})$ and all $i \in \N$. 
The assertion of the proposition follows from \eqref{eq:Cauchy_ineq_full}, as soon as we check that the quantity $A_c$  defined by 
\begin{multline}
    A_c:= \sum_{i=0}^{n-1}[c(\bar{x}_{i\tau}^{\tau/2^p},\bar{y}_{(i+1)\tau}^\tau)- c(\bar{x}_{(i+1)\tau}^{\tau/2^p},\bar{y}_{i\tau}^\tau) + c(\bar{x}_{i\tau}^\tau,\bar{y}_{(i+1)\tau}^{\tau/2^p})-c(\bar{x}_{(i+1)\tau}^\tau,\bar{y}_{i\tau}^{\tau/2^p})]\\
   \quad +2\sum_{i=0}^{n-1}c(\bar{x}_{(i+1)\tau}^\tau, \bar{y}_{i\tau}^\tau)+ 2\sum_{j=0}^{n2^p-1}c(\bar{x}_{(j+1)\tau/2^p}^{\tau/2^p}, \bar{y}_{j\tau/2^p}^{\tau/2^p})\label{eq:tph0}
\end{multline}
 is non-negative. 
Note that $A_c=A_{c_1}+A_{c_2}$. From the symmetry of $c_1\geq 0$ we obtain 
\begin{equation}
      A_{c_1}=2\sum_{i=0}^{n-1}c_1(\bar{x}_{(i+1)\tau}^\tau, \bar{y}_{i\tau}^\tau)+ 2\sum_{j=0}^{n2^p-1}c_1(\bar{x}_{(j+1)\tau/2^p}^{\tau/2^p}, \bar{y}_{j\tau/2^p}^{\tau/2^p}) \ge 0.\label{eq:tph1}
\end{equation}

Using the triangle inequality and the fact that $x_i^\tau = y_i^\tau $ and $x_i^{\tau/2^p} = y_i^{\tau/2^p} $, for  every $i=0, \dots, n-1$ one has that 
\begin{align*}
   c_2(\bar{x}_{(i+1)\tau}^\tau, \bar y_{i\tau}^\tau)+ c_2(\bar x_{i\tau}^\tau,\bar y_{(i+1)\tau}^{\tau/2^p})+c_2(\bar x_{(i+1)\tau}^{\tau/2^p}, \bar y_{i\tau}^{\tau/2^p}) &\ge c_2(\bar x_{(i+1)\tau}^\tau,\bar y_{i\tau}^{\tau/2^p}),\\
   c_2(\bar x_{(i+1)\tau}^{\tau/2^p}, \bar y_{i\tau}^{\tau/2^p})+c_2(\bar x_{i\tau}^{\tau/2^p},\bar y_{(i+1)\tau}^\tau) + c_2(\bar x_{(i+1)\tau}^\tau, \bar y_{i\tau}^\tau) &\ge c_2(\bar x_{(i+1)\tau}^{\tau/2^p},\bar y_{i\tau}^\tau).
\end{align*}
At the same time, again the triangle inequality gives 
\begin{equation*}
    2  \sum_{j=0}^{n2^p-1}c_2(\bar x_{(j+1)\tau/2^p}^{\tau/2^p}, \bar y_{j\tau/2^p}^{\tau/2^p}) \ge   2  \sum_{i=0}^{n-1}c_2(\bar x_{(i+1)\tau}^{\tau/2^p}, \bar y_{i\tau}^{\tau/2^p}).
\end{equation*}
By combining these inequalities one gets that 
\begin{multline}
   A_{c_2}\geq 
   \sum_{i=0}^{n-1}[c_2(\bar{x}_{i\tau}^{\tau/2^p},\bar{y}_{(i+1)\tau}^\tau)- c_2(\bar{x}_{(i+1)\tau}^{\tau/2^p},\bar{y}_{i\tau}^\tau) + c_2(\bar{x}_{i\tau}^\tau,\bar{y}_{(i+1)\tau}^{\tau/2^p})-c_2(\bar{x}_{(i+1)\tau}^\tau,\bar{y}_{i\tau}^{\tau/2^p})]\\
   +2\sum_{i=0}^{n-1}c_2(\bar{x}_{(i+1)\tau}^\tau, \bar{y}_{i\tau}^\tau)+ 2\sum_{i=0}^{n-1}c_2(\bar x_{(i+1)\tau}^{\tau/2^p}, \bar y_{i\tau}^{\tau/2^p}) \geq 0. 
   \label{eq:tph2}
\end{multline}
Inequalities \eqref{eq:tph1}--\eqref{eq:tph2} imply  that $A_c = A_{c_1} + A_{c_2} \geq 0$,  whence the thesis. 
\end{proof}

\begin{theorem}[Existence of EVI solutions  via   implicit scheme]\label{thm:existence_limit_scheme}  Let $c$ be a cost satisfying \eqref{eq:addass}  and let   Assumption \ref{ass:ass1} hold. Suppose   in  addition that 
\begin{enumerate}
    \item $g$ is lower bounded;
    \item there exist a topology $\sigma$ on $X$ such that $\sigma$ is compatible with $(c,g)$ and $c$ is additionally $\sigma$-continuous;
    \item Either
    \begin{enumerate}
        \item $\sigma$ is Cauchy-compatible with $(c,g)$ and $c$ decomposes as in Proposition \ref{prop:cauchy}, or
        \item   the sublevel sets $g$ are sequentially compact for $\sigma$.
    \end{enumerate}
\end{enumerate}
Then, for every $x_0 \in X$ and $\tau \in (0, \bar{\tau})$ the family of curves $(\bar{x}^{\tau/2^n})$ (constructed according to \eqref{eq:scheme} starting from $x_0$) converges (up to a subsequence in case (b)) pointwise (w.r.t.~$\sigma$) as $\tau \downarrow 0$ to a $\sigma$-continuous curve $x:[0,+\infty) \to X$ which may depend on $\tau$ and satisfies
\begin{equation}
c(x, x_t)-c(x,x_s) + \lambda\int_{s}^t c(x,x_r) \de r 
\le (t-s)g(x)-\int_s^t g(x_r) \de r \quad \forall \, 0 \le s \le t, \  x \in X.    \label{eq:evi}
\end{equation}
If $\EVI_\lambda(X,c,g)$  contains  a unique trajectory $x(\cdot)$ starting from $x_0$, then, $\bar{x}^{\tau}(\cdot)$ $\sigma$-converges pointwise to this $x(\cdot)$, which does not depend on $\tau$. In case (a), we  also have the following error estimate
\begin{equation}\label{eq:error_estimate}
    c \left (\bar{x}^{\tau}_t, x_t \right ) \le 2\tau (g(x_0)-\inf g) \quad \forall t \ge 0, \, \tau \in (0, \bar{\tau}).
\end{equation}
\end{theorem}
\begin{proof} We apply Lemma \ref{le:intevi}, written for  the time step  $\tau/2^n$. Discarding the second non-negative term in the  l.h.s.\ of \eqref{eq:intevi},  testing against any $x \in X$, we obtain, whenever $0 \le s \le t$ and $j,k \in \N$ are such that $j\tau \le s 2^n < (j+1)\tau$ and $k\tau \le t 2^n < (k+1)\tau$, that
\begin{align}\nonumber
&c(x,\bar{x}^{\tau/{2^n}}_t)-c(x,\bar{x}^{\tau/{2^n}}_s)  + \lambda \tau \sum_{i=j}^{k-1}c(x,x_{i+1}^{\tau/2^n}) + \tau \sum_{i=j}^{k-1} g(x_{i+1}^{\tau/2^n})\\ \nonumber 
    &\quad =c(x,\bar{x}^{\tau/{2^n}}_t)-c(x,\bar{x}^{\tau/{2^n}}_s)  + \int_{j\tau 2^{-n}}^{(k-1)\tau 2^{-n}} \left (\lambda c(x,\bar{x}^{\tau/2^n}_r) + g(\bar{x}^{\tau/2^n}_r) \right ) \de r \\[2mm] \label {eq:discevie3}
    &\quad \le (t-s) g(x).  
\end{align}
Notice that, choosing $x=\bar{x}^{\tau/2^n}_s$ and dropping the term $\lambda c\ge 0$, we get
\begin{equation}\label{eq:cocont}
c(\bar{x}^{\tau/2^n}_t, \bar{x}^{\tau/2^n}_s) \le (t-s)(g(\bar{x}^{\tau/2^n}_s)-A) \le (t-s)(g(x_0)-A) \quad \forall \, 0 \le s \le t, 
\end{equation}
where $A$ is a lower bound for $g$.
Now we split the proof in the two cases {\it (a)} and {\it (b)} as above.

In case {\it (a)}, by Proposition \ref{prop:cauchy}, we get that
    \begin{equation}\label{eq:mminn}
    c \left (\bar{x}^{\tau/{2^n}}_t, \bar{x}^{\tau/{2^m}}_t \right ) \le \frac{\tau (g(x_0)-A)}{2^{ m \wedge n-1}} \quad \forall t \ge 0, \, \tau \in (0, \bar{\tau}), \, m,n \in \N.
    \end{equation}
    This shows that the sequence $(\bar{x}^{\tau/{2^n}}_t)_{n \in \N}$ is $c$-Cauchy for every 
    $t\ge 0$ so that, by Cauchy-compatibility, it converges (w.r.t.~$\sigma$) as $n \to + \infty$ to a point $x_t \in X$. 

In case {\it (b)}, let us set
\[ \varphi_n(t):=\sup \left \{ \sum_{i=1}^N c(\bar{x}_{t_{i-1}}^{\tau/2^n}, \bar{x}_{t_{i}}^{\tau/2^n}) : 0=t_0 < t_1 < \cdots < t_{N-1} < t_N=t, \, N \in \N \right \}, \quad t \ge 0. \]
Clearly, $\varphi_n$ is a non-decreasing function such that $\varphi_n(t) \le t (g(x_0)-A)$ by \eqref{eq:cocont} for every $t \ge 0$, and
\begin{equation}\label{eq:helly}
c(\bar{x}_{t}^{\tau/2^n}, \bar{x}_{s}^{\tau/2^n}) \le \varphi_n(t) - \varphi_n(s)  \quad \text{ for every } 0 \le s \le t.   
\end{equation}
By Helly's theorem, up to passing to a unrelabeled subsequence, we can assume that there exists a non-decreasing function $\varphi:[0,+\infty) \to \R$ such that $\varphi_n(t) \to \varphi(t)$ for every $t \ge 0$.  Take  a countable dense subset $A$ of $[0,+\infty)$ such that $\varphi$ is continuous on $A^c$. By the diagonal principle and  the  sequential compactness of $\{x\, |\, g(x)\le g(x_0) \}$,  we can find a subsequence $n_k \uparrow + \infty$ and points $(x_t)_{t \in A} \subset X$ such that $\bar{x}_t^{\tau/2^{n_k}}$ $\sigma$-converges to $x_t$ as $k \to +\infty$. We now show that, for every  $t \in A^c$, we have that there exists a point $x_t \in X$ such that $\bar{x}_t^{\tau/2^{n_k}}$ $\sigma$-converges to $x_t$. Let $t \in A^c$ and let $n_{k_h} \uparrow + \infty$ and $u_t \in X$ be such that $\bar{x}_t^{\tau/2^{n_{k_h}}}$ $\sigma$-converges to $u_t$ as $h \to + \infty$. Let $(t_j)_j \subset A$ be such that $t_j \downarrow t$ as $j \to +\infty$ and let us write \eqref{eq:helly} for $t < t_j$,  namely, 
\[c(\bar{x}_{t_j}^{\tau/2^{n_{k_h}}}, \bar{x}_{t}^{\tau/2^{n_{k_h}}}) \le \varphi_{n_{k_h}}(t_j) - \varphi_{n_{k_h}}(t)  \quad \text{ for every } j,h \in \N.\]
Using the $\sigma$-continuity of $c$, we can pass to the limit as $h \to + \infty$ and we get
\[ c(x_{t_j}, u_t) \le \varphi(t_j) - \varphi(t).\]
By  the  continuity of $\varphi$ at $t$, we can pass to the limit as $j \to +\infty$ and we obtain that
\[ \lim_{j \to +\infty} c(x_{t_j}, u_t)=0.\]
Hence, by compatibility of $c$ with $\sigma$,  we have  that $u_t$ is the $\sigma$-limit of $x_{t_j}$ and does not depends on the chosen subsequence, thus identifying it as the limit of $\bar{x}_t^{\tau/2^{n_k}}$ as $k \to +\infty$.
This in particular proves the existence of $x:[0,+\infty) \to X$ such that $\bar{x}_t^{\tau/2^{n_k}}$ $\sigma$-converges to $x_t$ for every $t \ge 0$.

In both cases {\it (a)} and {\it (b)}, we can pass to the limit as $n_k \to + \infty$ in \eqref{eq:discevie3} and in \eqref{eq:cocont} and  get that $x(\cdot)$ satisfies \eqref{eq:evi} and,  by virtue of  the compatibility of $c$ with $\sigma$ and  of  \eqref{eq:cocont}, it is $\sigma$-continuous.
Assume that $\EVI_\lambda(X,c,g)$  contains  a unique trajectory $x(\cdot)$ starting from $x_0$. In case {\it (a)}, we just take $n=0$ and take $m\to+\infty$ in \eqref{eq:mminn} to derive \eqref{eq:error_estimate}.  Equation  \eqref{eq:error_estimate} also shows the pointwise convergence when taking the limit $\tau\to 0$. In case {\it (b)}, every converging subsequence w.r.t.\ $n$ of $(\bar{x}_t^{\tau/2^{n}})_{t\in A}$  converges  to a curve  in   $\EVI_\lambda(X,c,g)$, hence to $x(\cdot)$. This also shows that, for every $t>0$ and every neighborhood $V$ of $x_t$, there  exists   $N\in\N$ such that, for every $\tau\in (0,\frac{\bar \tau}{2^N})$, $\bar{x}_t^{\tau}\in V$.  This  concludes the proof.
\end{proof}

\subsection{Existence of the flow for splitting schemes}\label{sec:f+gproof}

In this section, we prove \Cref{thm:existence_limit_scheme_fg}. The main steps are nearly identical to the implicit case $f=0$. We refer to \Cref{table:scheme_fg} for the relations between the iterates.

\begin{lemma}\label{le:evi_f_fg} Under Assumption \ref{ass:ass2} the scheme in \eqref{eq:scheme2} satisfies for every $i \in  \N $ and every $0<\tau< \bar{\tau}$
\begin{multline}\label{eq:evidisc_f}
\frac{1}{\tau} [c(x,y_{i}^\tau)-c(x,\xi_{i}^\tau)]+\frac{1}{\tau} [c(x_{i}^\tau,\xi_{i}^\tau)-c(x_{i+1}^\tau,y_{i}^\tau)]+\lambda_f (c(x,\xi_{i}^\tau)-c(x_{i}^\tau,\xi_{i}^\tau))
\\ \leq f(x)-f(x_{i+1}^\tau) \quad \forall \, x \in X, \ \xi_{i}^\tau \in \mathcal{S}_c(x_{i}^\tau) 
\end{multline} 
 with  $y^\tau_i,x_{i}^\tau,\xi_{i}^\tau$ as per \eqref{eq:scheme2}. 

If we additionally assume that   $c$ satisfies \eqref{eq:addass},  i.e., $z_{i}^\tau\in \mathcal{R}_c(x_{i}^\tau)$ exists with $c(x_{i}^\tau,z_{i}^\tau)=0$, then we have $z_{i}^\tau\in \mathcal{S}_c(x_{i}^\tau)$ and
\begin{equation}\label{eq:evidisc_fg}
\frac{1}{\tau} [c(x,z_{i+1}^\tau)-c(x,z_{i}^\tau)]
        +\lambda_f c(x,z_{i}^\tau)+\lambda_g c(x,z_{i+1}^\tau)
    \le \phi(x)-\phi(x_{i+1}^\tau) \quad \forall \, x \in X,
\end{equation} 
with $\phi=f+g$.
\end{lemma}

\begin{proof}
    Fix $\xi_{i}^\tau\in \mathcal{S}_c(x_{i}^\tau)$ and $y^\tau_i$ defined as per \eqref{eq:scheme2}. By definition of the $c$-transform we can bound $f(x_{i+1}^\tau) \leq \frac{c(x_{i+1}^\tau,y_{i}^\tau)}{\tau}+f^{c/\tau}(y_{i}^\tau)$. Since $f$ is $c/\tau$-concave,  we get that   $f(x_{i}^\tau)=\frac{c(x_{i}^\tau,y_{i}^\tau)}{\tau}+f^{c/\tau}(y_{i}^\tau)$. Thus 
    \begin{equation}\label{eq:proof-fb-rates-1}
        f(x_{i+1}^\tau) \leq f(x_{i}^\tau)-\frac{c(x_{i}^\tau,y_{i}^\tau)}{\tau} + \frac{c(x_{i+1}^\tau,y_{i}^\tau)}{\tau}.
    \end{equation}
    The $\lambda_f \tau$-strongly $c/\tau$-cross-convexity of $f$ at $x_{i}^\tau$ gives
    \begin{equation}\label{eq:proof-fb-rates-3}
        f(x_{i}^\tau)\leq f(x)+\frac{1}{\tau} [c(x,\xi_{i}^\tau)-c(x_{i}^\tau,\xi_{i}^\tau)]-\frac{1}{\tau} [c(x,y_{i}^\tau)-c(x_{i}^\tau,y_{i}^\tau)]-\lambda_f (c(x,\xi_{i}^\tau)-c(x_{i}^\tau,\xi_{i}^\tau)).
    \end{equation}
    Summing \eqref{eq:proof-fb-rates-1} and \eqref{eq:proof-fb-rates-3}, we get \eqref{eq:evidisc_f}.
    
    On the other hand, by the $x$-update  (cf.~\eqref{eq:scheme2})  we have \begin{equation}\label{eq:proof-fb-rates-0}
        g(x_{i+1}^\tau) + \frac{c(x_{i+1}^\tau,y_{i}^\tau)}{\tau}\leq g(x_{i}^\tau)+ \frac{c(x_{i}^\tau,y_{i}^\tau)}{\tau}.
    \end{equation}
    Summing \eqref{eq:proof-fb-rates-1} and \eqref{eq:proof-fb-rates-0}  yields 
    \begin{equation}\label{eq:proof-fb-rates-2}
        f(x_{i+1}^\tau)+g(x_{i+1}^\tau)\leq f(x_{i}^\tau) +g(x_{i}^\tau),
    \end{equation}

    Using strong cross-concavity of $-g$ at $y_{i}^\tau$ gives us 
    \begin{multline}\label{eq:proof-fb-rates-4}
        g(x_{i+1}^\tau)\leq g(x)-\frac{1}{\tau}[c(x,z_{i+1}^\tau)-c(x_{i+1}^\tau,z_{i+1}^\tau)]+\frac{1}{\tau}[c(x,y_{i}^\tau)-c(x_{i+1}^\tau,y_{i}^\tau)]\\-\lambda_g(c(x,z_{i+1}^\tau)-c(x_{i+1}^\tau,z_{i+1}^\tau)).
    \end{multline}
    Since $c(x_{i}^\tau,z_{i}^\tau)=0$ and $c\ge 0$ by  \eqref{eq:addass},  we have that for all $x'\in X$, $c(x_{i}^\tau,z_{i}^\tau)=0\le c(x',z_{i}^\tau)$. So $z_{i}^\tau\in \mathcal{S}_c(x_{i}^\tau)$ and we can choose $\xi_{i}^\tau=z_{i}^\tau$.  In particular,   $y_{i}^\tau$ depends on $z_{i}^\tau$, and we can use $z_{i}^\tau$ in \eqref{eq:evidisc_f}.
    
    After summing \eqref{eq:evidisc_f} with $\xi_{i}^\tau=z_{i}^\tau$ and \eqref{eq:proof-fb-rates-4} and using that $c(x_{i}^\tau,z_{i}^\tau)=0$, we obtain \eqref{eq:evidisc_fg}.
\end{proof}
\begin{remark}[EVI for explicit/splitting scheme]
    If  $g=0$  we can take $\xi_{i}^\tau=y_{i-1}^\tau$  in \eqref{eq:evidisc_f} and prove that   $(f(x_{i}^\tau))_i$ is a decreasing sequence,  see  \eqref{eq:proof-fb-rates-1}. The term $[c(x_{i}^\tau,y_{i-1}^\tau)-c(x_{i+1}^\tau,y_{i}^\tau)]$ is then telescopic but has unknown sign. Nevertheless, if  \eqref{eq:addass}  holds, then this term vanishes since $c(x_{i+1}^\tau,y_{i}^\tau)=0$, and we obtain
   \begin{equation}
       \frac{1}{\tau} [c(x,y_{i}^\tau)-c(x,y_{i-1}^\tau)]+\lambda_f (c(x,\xi_{i}^\tau)-c(x_{i}^\tau,\xi_{i}^\tau)) \leq f(x)-f(x_{i+1}^\tau)
   \end{equation}
   which is a our candidate  discrete  EVI  for  the explicit scheme. 

    Our EVI \eqref{eq:evidisc_fg} for the splitting scheme differs from the EVI \eqref{eq:tobesum} of the implicit scheme as the term $c(x_{i+1}^\tau,y_{i}^\tau)$ is missing. This is because \eqref{eq:evidisc_fg} is obtained by adding up inequalities, among which the $c$-concavity of $f$, which for $f=0$ boils down to the uninformative $c(x_{i+1}^\tau,y_{i}^\tau)\ge 0$.
\end{remark}

\begin{lemma}\label{le:intevi_fg}  Let $c$ be a cost satisfying \eqref{eq:addass}  and let  Assumption \ref{ass:ass2} hold. Let  $\bar{x}^\tau$ and $\bar{z}^\tau$ be as in Definition \ref{def:interp2} for an arbitrary starting point $x_0 \in X$.
 We  have, for any $0 \le s= m\tau < n\tau = t$, that
\begin{equation}\label{eq:proto_evi_fg}
    \frac{c(x,\bar{z}_t^\tau)-c(x,\bar{z}_s^\tau)}{t-s} + \frac{\lambda_g \tau}{t-s} \sum_{i=m}^{n-1}c(x,z_{i+1}^\tau)+ \frac{\lambda_f \tau}{t-s} \sum_{i=m}^{n-1}c(x,z_{i}^\tau) \le \phi(x) -\phi(x_{n}^\tau) \quad \forall\, x \in X.
\end{equation}

\end{lemma}
\begin{proof}

 Using  Lemma \ref{le:evi_f_fg}, we get 
\begin{equation}\label{eq:tobesum_fg}
\frac{c(x,z_{i+1}^\tau)- c(x,z_i^\tau)}{\tau} + \lambda_g c(x,z_{i+1}^\tau)+ \lambda_f c(x,z_{i}^\tau) \le \phi(x)-\phi(x_{i+1}^\tau) \quad \forall \, x \in X \, i \in  \N .
\end{equation}
We sum \eqref{eq:tobesum_fg} from $i=m$ to $i=n-1$ to get
\[ \frac{c(x,\bar{z}_t^\tau)-c(x,\bar{z}_s^\tau)}{\tau} + \sum_{i=m}^{n-1}[\lambda_g c(x,z_{i+1}^\tau)+ \lambda_f c(x,z_{i}^\tau)] \le (n-m)\phi(x) - \sum_{i=m}^{n-1} \phi(x_{i+1}^\tau) \]
and dividing by $(n-m)$ we have
\begin{multline}\label{eq:proto-EVI_fg}
    \frac{c(x,\bar{z}_t^\tau)-c(x,\bar{z}_s^\tau)}{t-s} + \frac{\tau}{t-s} \sum_{i=m}^{n-1}[\lambda_g c(x,z_{i+1}^\tau)+ \lambda_f c(x,z_{i}^\tau)]\\
    \le \phi(x) - \frac{\tau}{t-s} \sum_{i=m}^{n-1} \phi(x_{i+1}^\tau)\le \phi(x) -\phi(x_{n}^\tau) ,  
\end{multline}
where we used that $\phi(x_{i+1}^\tau) \le \phi(x_i^\tau)$ for every $i \in  \N $.
\end{proof}

\begin{proposition}  Let $c$ be a cost satisfying \eqref{eq:addass}  and let  Assumption \ref{ass:ass2} hold. Assume   additionally  that  $\lambda_f\ge 0$, $\lambda_g\ge 0$, and  that  $c$  is symmetric. Then, for any $0 \le s= m\tau < n\tau=t$ and $p \in \N$, we have
\begin{equation}\label{eq:Cauchz_ineq_full_fg}
c(x_{t}^{\tau/2^p},z_{t}^\tau)
    \le \tau(\phi(x_0)-\phi(x_{n\tau}^\tau)) . 
\end{equation}
\end{proposition}
\begin{proof} Since $c\ge 0$ and $\lambda_f\ge 0$, $\lambda_g\ge 0$,  we can drop the $\lambda$-term in the l.h.s.\  of  
\eqref{eq:tobesum_fg} \eqref{eq:proto-EVI_fg}. We write these inequalities for two sequences, based on steps $\tau$ and $\tau/2^p$,  respectively, getting 
\begin{equation}\label{eq:tobesum_simple2_fg}
\frac{c(x,\bar{z}_{(i+1)\tau}^\tau)- c(x,\bar{z}_{i\tau}^\tau)}{\tau} \le \phi(x)-\phi(x_{(i+1)\tau}^\tau) \quad \forall \, x \in X,
\end{equation}
\begin{equation}\label{eq:proto-EVI_simple2_fg}
    \frac{c(u,z_{(i+1)\tau}^{\tau/2^p})-c(u,z_{i\tau}^{\tau/2^p})}{\tau} \le \phi(u) -\phi(x_{(i+1)\tau}^{\tau/2^p}) \quad \forall \, u \in X.
\end{equation}
 Summing  \eqref{eq:tobesum_simple2_fg} with $x=x_{(i+1)\tau}^{\tau/2^p}$ and \eqref{eq:proto-EVI_simple2_fg} with $u=x_{i\tau}^\tau$ we obtain
\begin{equation*}
    c(x_{(i+1)\tau}^{\tau/2^p},z_{(i+1)\tau}^\tau)- c(x_{(i+1)\tau}^{\tau/2^p},z_{i\tau}^\tau) + c(x_{i\tau}^\tau,z_{(i+1)\tau}^{\tau/2^p})-c(x_{i\tau}^\tau,z_{i\tau}^{\tau/2^p})
    \le \tau(\phi(x_{i\tau}^\tau)-\phi(x_{(i+1)\tau}^\tau)).
\end{equation*}
 By symmetry, we have  $c(x_{(i+1)\tau}^{\tau/2^p},z_{i\tau}^\tau)= c(x_{i\tau}^\tau,z_{(i+1)\tau}^{\tau/2^p})$,  so that summing  from $i=0\dots n-1$,  we   obtain
\begin{equation*}
    c(x_{t}^{\tau/2^p},z_{t}^\tau)
    \le \tau(\phi(x_0)-\phi(x_{n\tau}^\tau)).\qedhere
\end{equation*}
\end{proof}

 Starting from the discrete EVI \eqref{eq:proto_evi_fg} and using the Cauchy property in \eqref{eq:Cauchz_ineq_full_fg}, the proof of Theorem \ref{thm:existence_limit_scheme_fg} is then identical to  that  of Theorem \ref{thm:existence_limit_scheme}.

\section*{Acknowledgments} 
Most of the article was written while Pierre-Cyril
Aubin-Frankowski was hosted at TU Wien, being funded by the FWF project P 36344-N. The research of Ulisse Stefanelli and Giacomo Sodini was funded in whole or in part by the Austrian Science Fund (FWF) projects 10.55776/F65, 10.55776/I5149, 10.55776/P32788, as well as by the OeAD-WTZ project CZ  05/2024.   For open-access purposes, the authors have applied a CC BY public copyright license to any author-accepted manuscript version arising from this submission.

\newcommand{\etalchar}[1]{$^{#1}$}

\end{document}